\documentclass[a4paper,11pt,DIV=11,% decrease borders by 2 (std 10 for 11pt)
abstract=on% return to old style centered abstract
]{scrartcl}
\usepackage{mathtools}
\mathtoolsset{showonlyrefs}
\usepackage{amsmath, amsthm, amssymb}
\usepackage{fontenc}
\usepackage[utf8]{inputenc}
\usepackage[english]{babel}
\usepackage{graphicx}
\usepackage{subcaption,xargs}
\usepackage{algorithm, booktabs}
\usepackage[noend]{algpseudocode}
\usepackage{enumitem,listings,microtype}
\usepackage{hyperref}
\usepackage{environ,ifthen,xcolor}
\DeclareCaptionLabelSeparator{periodspace}{.\ }
\setkomafont{sectioning}{\rmfamily\bfseries}
\setkomafont{title}{\rmfamily}
\newcommand{\N}{\ensuremath{\mathbb{N}}}

\newcommand{\R}{\ensuremath{\mathbb{R}}}

\newcommand{\tT}{\mathrm{T}}

\DeclareMathOperator*{\argmin}{arg\,min}
\DeclareMathOperator{\prox}{prox}

\DeclareMathOperator{\vspan}{span}
\DeclareMathOperator{\Log}{Log}
\DeclareMathOperator{\Exp}{Exp}

\newtheorem{theorem}{Theorem}[section]

\newtheorem{proposition}[theorem]{Proposition}
\newtheorem{lemma}[theorem]{Lemma}
\newtheorem{remark}[theorem]{Remark}

\newtheorem{corollary}[theorem]{Corollary}

\hypersetup{pdfauthor={Miroslav Bačák,Ronny Bergmann, Gabriele Steidl, Andreas Weinmann},
	pdftitle={A Second Order Non-Smooth Variational Model for Restoring 
Manifold-Valued Images},
	pdfsubject={preprint},%
	pdfcreator = {pdflatex and TextMate},%
	pdfkeywords={},
	plainpages=false, pdfstartview=FitH, pdfview=FitH, pdfpagemode=UseOutlines,%
	bookmarksnumbered=true,bookmarksopen=false,bookmarksopenlevel=0,%
	colorlinks=true,linkcolor=black,citecolor=black,urlcolor=black%
}

\title{A Second Order Non-Smooth Variational Model\\ for Restoring 
Manifold-Valued Images}
\date{October 26, 2016}
\author{%
Miroslav Ba{\v c}{\'a}k\footnotemark[1]%
\and%
Ronny Bergmann\footnotemark[2]%
\and%
Gabriele Steidl\footnotemark[2]%
\and%
Andreas Weinmann\footnotemark[3]%
}
\begin{document}
\maketitle
\footnotetext[1]{Max Planck Institute for Mathematics in the Sciences, Inselstr.~22, 04103 Leipzig, Germany, bacak@mis.mpg.de}
\footnotetext[2]{Department of Mathematics,
Technische Universit\"at Kaiserslautern,
Paul-Ehrlich-Str.~31, 67663 Kaiserslautern, Germany,
$\{$bergmann, steidl$\}$@mathematik.uni-kl.de.}
\footnotetext[3]{Department of Mathematics,
Technische Universit\"at M\"unchen and
Fast Algorithms for Biomedical Imaging Group, Helmholtz-Zentrum M\"unchen,
Ingolst\"adter Landstr.~1, 85764 Neuherberg, Germany,
andreas.weinmann@tum.de.}

%-----------------------------------------------------------------------------
\begin{abstract}
	\noindent\small
	We introduce a new non-smooth variational model for the restoration of
	manifold-valued data which includes second order differences in the
	regularization term. 
	While such models were successfully applied for real-valued
	images, we introduce the second order difference and the corresponding variational models for manifold data, which up to now only existed for cyclic data.
	The approach requires a combination of techniques from numerical analysis, 
	convex optimization and differential geometry.
	First, we establish a suitable definition of absolute second order
	differences for signals and images with values in a manifold.
	Employing this definition, we introduce a variational denoising model
	based on first and second order differences in the manifold setup.
	In order to minimize the corresponding functional,
	we develop an algorithm using an inexact cyclic proximal point algorithm.
	We propose an efficient strategy for the computation of the corresponding
	proximal mappings 
	in symmetric spaces
	utilizing the machinery of Jacobi fields.
	For the $n$-sphere
	and the manifold of symmetric positive definite matrices,
	we demonstrate the performance of our algorithm in practice.
	We prove the convergence of the proposed exact and inexact variant of the
	cyclic proximal point algorithm
	in Hadamard spaces. These results which are of interest on its own include,
	e.g., the manifold of symmetric positive definite matrices.
\end{abstract}
\pagestyle{myheadings} 
\thispagestyle{plain}
%
% \markboth{\textsc M.~Ba{\v c}{\'a}k,
% R.~Bergmann,
% G.~Steidl,
% A.~Weinmann}%
% {\textsc A Second Order Non-Smooth Variational Model for Restoring Manifold-Valued Images}
%
\section{Introduction}\label{sec:intro}
\enlargethispage{2\baselineskip}
%-------------------------------------------------------------------------------
%
In this paper, we introduce a non-smooth variational model for the restoration
of manifold-valued images using first and second order differences.
The model can be seen as a second order generalization of the
Rudin-Osher-Fatemi (ROF) functional~\cite{ROF92}
for images taking their values in a Riemannian manifold.
For scalar-valued images, the ROF functional in its discrete, anisotropic
penalized form is given by
\[
\tfrac12 \lVert f - u \rVert_2^2 + \lambda  \lVert \nabla u \lVert_1
, \quad \lambda > 0,
\]
where $f \in \mathbb R^{N,M}$ is a given noisy image and the symbol $\nabla$ is
used to denote the discrete first order difference operator which usually
contains the forward differences in vertical and horizontal directions.
The frequently used ROF denoising model preserves important image structures as
edges, but tends to produce staircasing: instead of reconstructing smooth areas
as such, the reconstruction consists of constant plateaus with small jumps.
An approach for avoiding this effect incorporates higher order differences,
respectively derivatives, in a continuous setting.
The pioneering work~\cite{CL97} couples the TV term with higher order terms by
infimal convolution.
Since then, various techniques with higher order differences/derivatives were
proposed in the literature, among
them~\cite{BKP09,CMM00,DSS09,DWB09,HS06,LBU2012,LLT03,LT06,PS13,Sche98,SS08,SST11}.
We further note that the second-order total generalized variation was
extended for tensor fields in~\cite{VBK13}.

In various applications in image processing and computer vision the functions
of interest take values in a Riemannian manifold. 
One example is diffusion tensor imaging where the data lives in the Riemannian
manifold of positive definite matrices; see,
e.g.,~\cite{basser1994mr,BWFW04,pennec2006riemannian,SSPB07,WFWBB06,WFBW03}.
Other examples are color images based on non-flat color
models~\cite{chan2001total,kimmel2002orientation, lai2014splitting,vese2002numerical}
where the data lives on spheres.
Motion group and $\mathrm{SO}(3)$-valued data play a role in tracking, robotics
and (scene) motion analysis and were considered,
e.g., in~\cite{FB12,PBP95,RS10,rosman2012group,TPM08}.
Because of the natural appearance of such nonlinear data spaces,
processing manifold-valued data has gained a lot of interest
in applied mathematics in recent years.
As examples, we mention wavelet-type multiscale
transforms~\cite{GW09,RDSDS05,Wein12}%
, robust principal component pursuit on manifolds~\cite{HiTa2014},
and partial differential
equations~\cite{chefd2004regularizing, GHS13,tschumperle2001diffusion}
for manifold-valued functions.
Although statistics on Riemannian manifolds is not in the focus of this work,
we want to mention that, in recent years, there are many papers on this topic.

In~\cite{GM06,GM07}, the notion of total variation of functions
having their values on a manifold was investigated based on
the theory of Cartesian currents.
These papers extend the previous work~\cite{GMS93} where
circle-valued functions were considered.
The first work which applies a TV approach of circle-valued data
for image processing tasks is~\cite{SC11,CS13}. An
algorithm for TV regularized minimization problems on Riemannian
manifolds was proposed  in~\cite{LSKC13}. There, 
the problem is reformulated as a multilabel optimization problem
which is approached using convex relaxation techniques.
Another approach to TV minimization for manifold-valued data which
employs cyclic and parallel proximal point algorithms
and does not require labeling and relaxation techniques was given
in~\cite{WDS2014}.
In a recent approach~\cite{GS14} the restoration of manifold-valued images
was done using a smoothed TV model and an iteratively reweighted least squares
technique.
A method which circumvents the direct work with manifold-valued data by
embedding the matrix manifold in the appropriate Euclidean space and applying
a back projection to the manifold was suggested in~\cite{RTKB14}.
This can be also extended to higher order derivatives since the derivatives (or
differences) are computed in the Euclidean space.

Our paper is the next step in a program already consisting of a considerable
body of work of the authors:
in~\cite{WDS2014}, variational models using first order differences for general
manifold-valued data were developed.
In~\cite{BLSW14}, variational models using first and second order differences
for circle-valued data were introduced.
Using a suitable definition of second order differences on the
circle $\mathbb S^1$
the authors incorporate higher order differences into
the energy functionals to improve the denoising results for circle-valued
data. Furthermore, convergence for locally nearby data is shown.
Our paper~\cite{BW15b} extends this approach to
product spaces of arbitrarily many circles and a vector space,
and~\cite{BW15a} to inpainting problems.
Product spaces are important for example when dealing with nonlinear color
spaces such as HSV.

This paper continues our recent work considerably by generalizing the combined first and second order
variational models to general symmetric Riemannian manifolds. Besides cyclic data this includes
general $n$-spheres, hyperbolic spaces, symmetric positive definite matrices as well as compact Lie groups and Grassmannians. 
First we provide a novel definition of absolute second order differences for data with values in a
manifold.
The definition is geometric and particularly appealing since it avoids using
the tangent bundle
for its definition. As a result, it is computationally accessible by
the machinery of Jacobi fields which, in particular, in symmetric spaces
yields rather explicit descriptions -- even in this generality.
Employing this definition, we introduce a variational model for denoising
based on first and second order differences in the Riemannian manifold setup.
In order to minimize the corresponding functional, we
follow~\cite{BLSW14,BW15b,WDS2014} and use a cyclic proximal point algorithm
(PPA). In contrast to the aforementioned references,
in our general setup, no closed form expressions are available for some of the
proximal mappings involved.
Therefore, we use as approximate strategy, a subgradient descent to compute
them.
For this purpose, we derive an efficient scheme. We show
the convergence of the proposed exact and
inexact variant of the cyclic PPA in a Hadamard space.
This extends a result from~\cite{Bac14}, where the exact cyclic PPA in Hadamard spaces was proved to converge under more restrictive assumptions. 
Note that the basic (batch) version of the PPA in Hadamard spaces was introduced in~\cite{Bac13}. 
Another related result is due to S. Banert~\cite{Ba14}, who developed both exact and inexact PPA for a regularized sum of two functions on a product of Hadamard spaces. 
In the context of Hadamard manifolds, the convergence of an inexact proximal point method for multivalued vector fields was studied in~\cite{WLLY14}.

In this paper we prove the convergence of the (inexact) cyclic PPA under the general assumptions
required by our model which differs from the cited papers.

Our convergence statements apply in particular to
the manifold of symmetric positive definite matrices.
Finally, we demonstrate the performance of our algorithm in numerical
experiments
for denoising of images with values in spheres as well as in the space of symmetric positive definite matrices.

Our main application examples, namely $n$-spheres and manifolds of symmetric
positive definite matrices 
are, with respect to the sectional curvature, two extreme instances of symmetric
spaces.
The spheres have positive constant curvature, whereas the symmetric positive
definite matrices are non-positively curved.
Their geometry is totally different, e.g., in the  manifolds of symmetric
positive definite matrices
the triangles are slim and there are no cut locus which means that geodesics
are always shortest paths.
In $n$-spheres however, every geodesic meets a cut point and triangles are
always fat, meaning that the sum of the interior angles is always bigger than
$\pi$.
In our setup,
however, it turns out that the sign of the sectional curvature is not
important, but the important thing is the structure provided by symmetric
spaces. 

The outline of the paper is as follows:
We start by introducing our variational restoration model in
Section~\ref{sec:model}.
In Section~\ref{sec:order_two} we show how the (sub)gradients of
the second order difference operators can be computed. Interestingly, this
can be done by solving appropriate Jacobi equations.
We describe the computation for general symmetric spaces. Then we focus
on $n$-spheres and the space of symmetric positive definite matrices.
The (sub)gradients are needed within our inexact cyclic PPA which is proposed
in Section~\ref{sec:cppa}.
A convergence analysis of the exact and inexact cyclic PPA is given for
Hadamard manifolds.
In Section~\ref{sec:numerics} we validate our model and illustrate the good
performance of our algorithms
by numerical examples. The appendix provides some useful formulas
for the computations. 
Further, Appendix C gives a brief introduction into the concept
of parallel transport on manifolds in order to make our results better accessible 
for non-experts in differential geometry.
%------------------------------------------------------------------------------
\section{Variational Model} \label{sec:model}
%-------------------------------------------------------------------------------
{\fontdimen2\font=1.3ex Let $\mathcal M$ be a complete $n$-dimensional Riemannian manifold with
Riemannian} metric~$\langle\cdot,\cdot\rangle_x\colon
	T_x{\mathcal M}\times T_x\mathcal M \rightarrow \mathbb R$, induced 
	norm \(\lVert\cdot\rVert_x\), and geodesic
	distance~$d_{\mathcal M}\colon \mathcal M \times \mathcal M
	\rightarrow \mathbb R_{\ge 0}$.
Let $\nabla_{\mathcal M} F$ denote the Riemannian gradient 
of~$F\colon\mathcal M\rightarrow \mathbb R$ which is characterized for all $\xi \in T_x {\mathcal M}$ by
\begin{equation}\label{R_grad}
\langle \nabla_{\mathcal M} F (x) , \xi \rangle_x = D_x F [\xi],
\end{equation}
where \(D_x F\) denotes the differential of \(F\) at $x$, see Appendix \ref{C-transport}.

Let \(\gamma_{x,\xi}(t)\), \(x\in\mathcal M\), \(\xi\in T_x\mathcal M\) 
be
the unique geodesic starting from \(\gamma_{x,\xi}(0) = x\) with
\(\dot\gamma_{x,\xi} (0) = \xi\). 
Further let \(\tilde\gamma_{\overset{\frown}{x,z}}\) denote
a unit speed geodesic connecting~\(x,z\in\mathcal M\). 
Then it fulfills~\(\tilde\gamma_{\overset{\frown}{x,z}}(0) = x\),
\(\tilde\gamma_{\overset{\frown}{x,z}}(\mathcal L) = z\), where
\(\mathcal L = \mathcal L(\tilde\gamma_{\overset{\frown}{x,z}})\) denotes the length of the geodesic.
We further denote by \(\gamma_{\overset{\frown}{x,z}}\) a minimizing geodesic, 
i.e.\ a geodesic having minimal length \(\mathcal L(\tilde\gamma_{\overset{\frown}{x,z}}) = d_{\mathcal M}(x,z)\).
If it is clear from the context, we write geodesic instead of minimizing geodesic, but keep the notation 
of using \(\tilde\gamma_{\overset{\frown}{x,z}}\) when referring to all geodesics including the non-minimizing ones.
We use the exponential map \(\exp_x\colon T_x\mathcal M \rightarrow \mathcal M\) given by
\(\exp_x\xi = \gamma_{x,\xi}(1)\) and  
the inverse exponential map
denoted by \(\log_x = \exp_x^{-1}\colon \mathcal M \to T_x\mathcal M\). 
%-------------------------------------------------------------------------------
\begin{figure}
	\begin{subfigure}{.48\textwidth}\centering
		\includegraphics[width=.66\textwidth]{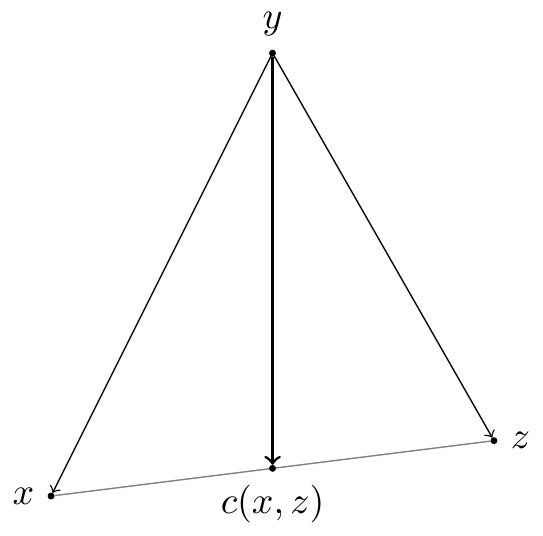}
		\caption{On \(\mathbb R^d\).}
		\label{subfig:2DiffR}
	\end{subfigure}
	\begin{subfigure}{.48\textwidth}\centering
		\includegraphics[width=.66\textwidth]{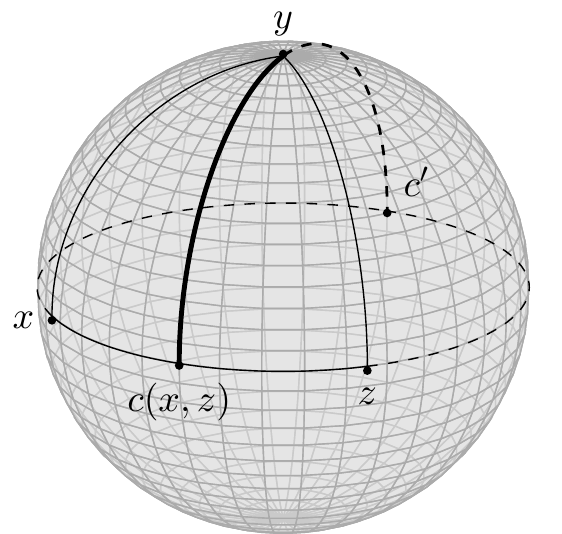}
		\caption{On \(\mathbb S^2\).}\label{subfig:2DiffS}
	\end{subfigure}
	\caption{Illustration of the absolute second order difference on
	\subref{subfig:2DiffR} the Euclidean space \(\mathbb R^n\) and
	\subref{subfig:2DiffS} the sphere \(\mathbb S^2\). In both cases the second order difference is the length of the line connecting \(y\) and \(c(x,z)\). 
	Nevertheless on \(\mathbb S^2\) there is a second minimizer \(c'\), which is the mid point of the longer arc of the great circle defined by \(x\) and \(z\).}
\end{figure}
%-------------------------------------------------------------------------------
%
The core of our restoration model are absolute second order differences of points lying in a manifold.
In the following we define such differences in a sound way.
The basic idea is based on rewriting the Euclidean norm of componentwise second order differences in $\mathbb R^n$ as
\(\lVert x-2y+z\rVert_2 = 2 \lVert \tfrac{1}{2}(x+z)-y\rVert_2\), see Fig.~\ref{subfig:2DiffR}.
We define the set of midpoints between $x,z \in {\mathcal M}$ as 
\[
\mathcal C_{x,z} \coloneqq
	\bigl\{ c\in \mathcal M : c = 
	\tilde\gamma_{\overset{\frown}{x,z}}\bigl(\tfrac{1}{2}\mathcal L(\tilde\gamma_{\overset{\frown}{x,z}})\bigr) \text{ for any geodesic } \tilde\gamma_{\overset{\frown}{x,z}}
\bigr\}
\]
and the \emph{ absolute second difference operator} $\mathrm{d}_2\colon {\mathcal M}^3 \rightarrow \mathbb R_{\ge 0}$
by
\begin{equation} \label{def_two_diff}
\mathrm{d}_2(x,y,z) \coloneqq
	\min_{c\,\in\,\mathcal C_{x,z}} d_{\mathcal M} (c,y),%
	\quad x, y, z \in \mathcal M.
\end{equation}
The definition is illustrated for \(\mathcal M = \mathbb S^2\) in Fig.~\ref{subfig:2DiffS}.
For the manifold $\mathcal M = \mathbb S^1$, definition~\eqref{def_two_diff} coincides, up to the factor $\frac12$, with those 
of the absolute second order differences in~\cite{BLSW14}.%
\enlargethispage{\baselineskip}
Similarly we define the \emph{second order mixed differences} $\mathrm{d}_{1,1}$ based on
$\lVert w - x + y - z\rVert_2 = 2\lVert\frac12(w+y) - \frac12 (x+z) \rVert_2$
as
\begin{equation} \label{def_sec_order_mix}
 \mathrm{d}_{1,1}(w,x,y,z) \coloneqq \min_{c \in {\mathcal C}_{w,y}, \tilde c  \in {\mathcal C}_{x,z}} d_{\mathcal M}(c,\tilde c), \qquad x,y,z,w \in {\mathcal M}.
\end{equation}

Let ${\mathcal G} \coloneqq \{1,\ldots,N\} \times \{1,\ldots,M\}$.
We want to denoise manifold-valued images 
$f\colon {\mathcal G} \rightarrow {\mathcal M}$
by minimizing functionals of the form
\begin{align}\label{task_2}
{\mathcal E}(u) = {\mathcal E}(u,f) \coloneqq
	F(u; f)
	+ \alpha \operatorname{TV}_1(u)
	+ \beta \operatorname{TV}_2(u)
	\end{align}
where
\begin{align}
	F(u;f)
	&\coloneqq
		\frac{1}{2}
		\sum_{i,j=1}^{N,M} d_{\mathcal M} (f_{i,j},u_{i,j})^2,\label{eq:2DFdist}
\\
	\alpha\operatorname{TV}_1(u)
	&\coloneqq
		\alpha_1\sum_{i,j=1}^{N-1,M} d_{\mathcal M}(u_{i,j},u_{i+1,j})
		+
		\alpha_2\sum_{i,j=1}^{N,M-1} d_{\mathcal M}(u_{i,j},u_{i,j+1}),
	\label{eq:2DTV}
\\
	\beta\operatorname{TV}_2(u)
	&\coloneqq
	\beta_1\sum_{i=2,j=1}^{N-1,M}
	\mathrm{d}_2(u_{i-1,j},u_{i,j},u_{i+1,j})
	+
	\beta_2\sum_{i=1,j=2}^{N,M-1}
	\mathrm{d}_2(u_{i,j-1},u_{i,j},u_{i,j+1})\\
	&\qquad
	+ \beta_3 \sum_{i,j=1}^{N-1,M-1}
		\mathrm{d}_{1,1}(u_{i,j},u_{i,j+1},u_{i+1,j},u_{i+1,j+1}).
\end{align}
For minimizing the functional we want to apply a cyclic PPA \cite{Bac14,bertsekas}.
This algorithm sequentially computes the proximal mappings of the summands involved in the functional.
While the proximal mappings of the summands in data term $F(u;f)$ 
and in the first regularization term $\operatorname{TV}_1(u)$ are known analytically, see~\cite{FO02} and~\cite{WDS2014}, respectively, 
the proximal mappings of 
$\mathrm{d}_2\colon {\mathcal M}^3 \rightarrow \mathbb R_{\geq0}$
and 
$\mathrm{d}_{1,1}\colon {\mathcal M}^4 \rightarrow \mathbb R_{\geq0}$
are only known analytically in the special case ${\mathcal M} = \mathbb S^1$, see~\cite{BLSW14}.
In the following section we deal with the computation of the proximal mapping of
$\mathrm{d}_2$. The difference $\mathrm{d}_{1,1}$ can be treated in a similar way.
%
%-------------------------------------------------------------------------------
\section{Subgradients of Second Order Differences on \texorpdfstring{$\mathcal M$}{M}} \label{sec:order_two}
%-------------------------------------------------------------------------------
%
Since we work in a Riemannian manifold, it is necessary to impose an assumption that guarantees 
that the involved points do not take pathological (practically irrelevant) constellations
to make the following derivations meaningful.
In particular, 
we assume in this section that there is exactly one shortest geodesic joining $x,z,$ 
i.e., $x$ is not a cut point of $z$; cf.~\cite{do1992riemannian}.  
We note that this is no severe restriction since 
such points form a set of measure zero.
Moreover, we restrict our attention to the case where the minimizer in~\eqref{def_two_diff}
is taken for the corresponding geodesic midpoint which we denote by $c(x,z)$.

We want to compute the proximal mapping of $\mathrm{d}_2\colon {\mathcal M}^3 \rightarrow \mathbb R_{\ge 0}$ 
by a (sub)gradient descent algorithm. 
This requires the computation of the (sub)gradient of $\mathrm{d}_2$ which is done in the following subsections.
For $c(x,z) \not = y$ the subgradient of $\mathrm{d}_2$ coincides with its gradient
\begin{equation} \label{abl}
\nabla_{ {\mathcal M}^3 } \mathrm{d}_2 = 
\left( \nabla_{\mathcal M} \mathrm{d}_2 (\cdot,y,z), \nabla_{\mathcal M} \mathrm{d}_2 (x,\cdot,z), \nabla_{\mathcal M} \mathrm{d}_2 (x,y,\cdot) \right)^\tT.
\end{equation}
If $c(x,z) = y,$ then $\mathrm{d}_2$ is not differentiable. However,  
we will characterize the subgradients in Remark~\ref{rem:subdiff}.
In particular, the zero vector is a subgradient which is used
in our subgradient descent algorithm.
%
%-------------------------------------------------------------------------------
\subsection{Gradients of the Components of Second Order Differences} \label{sec:order_two_grad}
%-------------------------------------------------------------------------------
%
We start with the computation of the second component of the
gradient~\eqref{abl}.
In general we have for $d_{\mathcal M} (\cdot,y)\colon\mathcal M \rightarrow \mathbb R_{\ge 0}$, $x \mapsto  d_{\mathcal M} (x,y)$, see \cite{TAV13},
that 
\begin{align} \label{deriv_d}
\nabla_{\mathcal M} d_{\mathcal M}^2 (x,y) = - 2 \log_x y, 
\quad
\nabla_{\mathcal M} d_{\mathcal M} (x,y) = -\frac{\log_x y}{\lVert\log_x y\rVert_x}, \; x \not = y .
\end{align}
%
%-------------------------------------------------------------------------------
\begin{lemma} \label{comp_y}
The second component of $\nabla_{ {\mathcal M}^3 } \mathrm{d}_2$ in~\eqref{abl} is given for $y \neq c(x,z)$~by
\begin{equation} \label{first}
\nabla_{\mathcal M} \mathrm{d}_2 (x,\cdot,z)(y) =  \frac{\log_y c(x,z)}{\lVert\log_y c(x,z)\rVert_y}.
\end{equation}
\end{lemma}

%-------------------------------------------------------------------------------
\begin{proof}
Applying~\eqref{deriv_d} for $\mathrm{d}_2 (x,\cdot,z) = d_{\mathcal M} (c(x,z),\cdot)$ we obtain
the assertion.
\end{proof}
%-------------------------------------------------------------------------------

By the symmetry of $c(x,z)$ both gradients 
$\nabla_{\mathcal M} \mathrm{d}_2 (\cdot,y,z)$ and $\nabla_{\mathcal M} \mathrm{d}_2(x,y,\cdot)$ 
can be realized in the same way so that we can restrict our attention to the first one.
For fixed $z \in {\mathcal M}$ we will use the notation $c(x)$ instead of $c(x,z)$.
Let $\gamma_{\overset{\frown}{x,z}} = \gamma_{x,\upsilon}$ 
denote the unit speed geodesic joining $x$ and $z$, i.e.,
\[
\gamma_{x,\upsilon} (0) = x, \ \dot \gamma_{x,\upsilon} (0) = \upsilon = \tfrac{\log_x z}{\lVert \log_x z \rVert_x}
\quad \text{and} \quad 
\gamma_{x,\upsilon} (T) = z,\ T\coloneqq d_{\mathcal M}(x,z).
\]
We will further need the notation of parallel transport.
For readers which are not familiar with this concept we give a brief introduction in Appendix \ref{C-transport}.
We denote a parallel transported orthonormal frame along $\gamma_{x,\upsilon}$ by
\begin{equation} \label{frames_par}
\{
\Xi_1 = \Xi_1(t), \ldots,\Xi_n = \Xi_n(t).
\}
\end{equation}
For $t=0$ we use the special notation $\{ \xi_1 , \ldots, \xi_n \}$.

%-------------------------------------------------------------------------------
\begin{lemma} \label{lem2}
The first component of $\nabla_{{\mathcal M}^3}\mathrm{d}_2$ in~\eqref{abl} is given for $c(x,z) \not = y$ by
\begin{equation} \label{important}
	\nabla_{\mathcal M} \mathrm{d}_2 (\cdot,y,z)(x) = \sum_{k=1}^n  \left\langle 
\frac{ \log_{c(x)} y }{ \rVert \log_{c(x)} y \lVert_{c(x)}} , D_x c[\xi_k]
\right\rangle_{c(x)} \xi_k.
\end{equation}
\end{lemma}
%-------------------------------------------------------------------------------

\begin{proof}
For $F\colon{\mathcal M} \rightarrow \mathbb R$ defined by
$F \coloneqq \mathrm{d}_2 (\cdot,y,z)$ 
we are looking for the coefficients $a_k = a_k(x)$ in
\begin{equation} \label{grad_a}
	\nabla_{\mathcal M} F(x) = \sum_{k=1}^n a_k \xi_k.
\end{equation}
For any tangential vector $\eta \coloneqq \sum_{k=1}^n \eta_k \xi_k \in T_x{\mathcal M}$ we have
\begin{equation} \label{grad_b}
	\langle \nabla_{\mathcal M} F(x), \eta \rangle_x = D_x F[\eta] = \sum_{k=1}^n \eta_k a_k.
\end{equation}
Since
$	F = f \circ c$,
with
$
f\colon {\mathcal M} \rightarrow \R,
\ 
x \mapsto d_{\mathcal M} (x,y)
$
we obtain by the chain rule 
\[
D_x F [\eta]   = \left( D_{c(x)} f \circ D_x c \right) [\eta].
\]
Now the differential of $c$ is determined by
\[
D_x c [\eta] = \sum_{k=1}^n \eta_k D_x c [\xi_k] \in T_{c(x)} {\mathcal M}.
\]
Then it follows by \eqref{R_grad} and \eqref{deriv_d} that
\begin{align*}
D_xF[\eta] &= D_{c(x)} f \big[ D_x c [\eta] \big] 
= \langle \nabla_{\mathcal M} f (c(x)), D_x c [\eta] \rangle_{c(x)}
\\
&= 
\left\langle  
-\frac{ \log_{c(x)} y }{ \rVert \log_{c(x)} y \lVert_{c(x)}} , \sum_{k=1}^n \eta_k D_x c[\xi_k]
\right\rangle_{c(x)},
\end{align*}
and consequently
\begin{equation} 
\langle \nabla_{\mathcal M} F(x), \eta \rangle_x
= \sum_{k=1}^n \eta_k \left\langle  
-\frac{ \log_{c(x)} y }{ \rVert \log_{c(x)} y \lVert_{c(x)}} , D_x c[\xi_k]
\right\rangle_{c(x)}.
\end{equation}
By~\eqref{grad_b} and~\eqref{grad_a} we obtain the assertion~\eqref{important}.
\end{proof}

\begin{figure}\centering
	\includegraphics{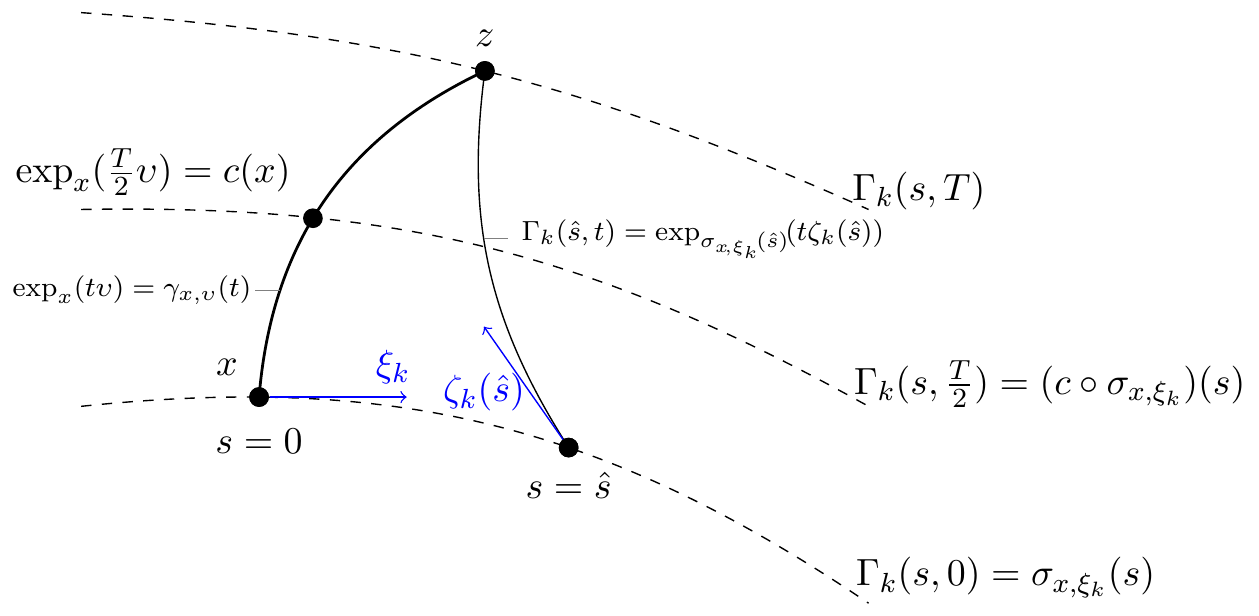}
	\caption{Illustration of the variations \(\Gamma_k (s,t)\) of the geodesic \(\gamma_{x,\upsilon}\) to define the Jacobi field \(J_k(t) = \frac{\partial}{\partial s}\Gamma_k(s,t)\bigr\rvert_{s=0}\) along \(\gamma_{x,\upsilon}\) with respect to \(\xi_k\).}\label{fig:jacobi}
\end{figure}
%-------------------------------------------------------------------------------
For the computation of $D_x c[\xi_k]$ we can exploit Jacobi fields which are defined as follows:
for $s \in (-\varepsilon,\varepsilon)$, let $\sigma_{x,\xi_k} (s)$, $k=1,\ldots,n$, denote the unit speed geodesic with~$\sigma_{x,\xi_k} (0) = x$ and
$\sigma_{x,\xi_k}' (0) = \xi_k$. 
Let $\zeta_k(s)$ denote the tangential vector in $\sigma_{x,\xi_k}(s)$ of the geodesic joining $\sigma_{x,\xi_k}(s)$ and $z$.
Then
\[
	\Gamma_k (s,t) \coloneqq
	\exp_{\sigma_{x,\xi_k} (s)} \left(t \zeta_k(s) \right),
		\quad s \in (- \varepsilon,\varepsilon), t \in [0,T],
\]
with small $\varepsilon > 0$ are variations of the geodesic $\gamma_{x,\upsilon}$
and
\[
J_k (t) \coloneqq \frac{\partial }{\partial s}  \Gamma_k (s,t)\bigr\rvert_{s=0} , \quad k=1,\ldots,n,
\]
are the corresponding Jacobi field along $\gamma_{x,\upsilon}$. For an illustration see Fig. \ref{fig:jacobi}.
Since $\Gamma_k (s,0) = \sigma_{x,\xi_k}(s)$ and $\Gamma_k (s,T) = z$ for all $s \in (-\varepsilon,\varepsilon)$ 
we have 
\begin{equation} \label{jacobi_boundary}
J_k (0) = \xi_k, \; J_k (T) = 0, \quad \quad k=1,\ldots,n.
\end{equation}
Since $\Gamma_k(s,\frac{T}{2}) = (c \circ \sigma_{x,\xi_k}) (s)$ we conclude by definition \eqref{def:differential} of the differential
\[
D_x c [\xi_k] = \frac{d}{ds} (c \circ \sigma_{x,\xi_k})(s)\bigr\rvert_{s=0} = \frac{d}{ds} \Gamma_k(s,\tfrac{T}{2})\bigr\rvert_{s=0} = J_k \left( \tfrac{T}{2} \right), \quad k=1,\ldots,n.
\]
Any Jacobi field $J$ of a variation through $\gamma_{x,\upsilon}$ fulfills 
a linear system of ordinary differential equations (ODE) \cite[Theorem 10.2]{Lee97}
\begin{equation} \label{jacobi_diff} 
 \frac{D^2}{\mathrm{d} t^2} J + R(J, \dot \gamma_{x,\upsilon}) \dot \gamma_{x,\upsilon}  = 0,
\end{equation}
where $R$ denotes the Riemannian curvature tensor
defined by 
$(\xi,\zeta,\eta) \to R(\xi,\zeta) \eta:= \nabla_\xi \nabla_\zeta \eta - \nabla _\zeta \nabla_\xi \eta - \nabla_{[\xi,\zeta]} \eta$.
Here $[\cdot,\cdot]$ is the Lie bracket,
see Appendix~\ref{C-transport}.
Our special Jacobi fields have to meet the boundary conditions~\eqref{jacobi_boundary}.
We summarize: 
%
%-----------------------------------
\begin{lemma}\label{jacobi}
The vectors $D_x c [\xi_k]$, $k=1,\ldots,n$, in~\eqref{important} are given by
\begin{equation} \label{relation}
D_x c [\xi_k] = J_k \left( \tfrac{T}{2} \right), \quad k=1,\ldots,n,
\end{equation}
where $J_k$ are the
Jacobi fields given by
\[
\frac{D^2}{\mathrm{d} t^2} J_k + R(J_k,\dot \gamma_{x,\upsilon}) \dot \gamma_{x,\upsilon} = 0, \quad J_k (0) = \xi_k, \; J_k (T) = 0.
\]
\end{lemma}
%-------------------------------------------------------------------------------

%
Finally, we give a representation of the subgradients of $\mathrm{d}_2$ in the case $c(x,z) = y$.

%
%-------------------------------------------------------------------------------
\begin{remark}\label{rem:subdiff}
Let  $(x,y,z) \in {\mathcal M}^3$ with $c(x,z) = y$ and 
\begin{equation}\label{rem:subdiffAux}      
\mathcal V \coloneqq 
      \left\{	
      \xi \coloneqq \left(\xi_x, \xi_y ,\xi_z \right) :  
     	 \xi_x \in T_x \mathcal M,
     	 \xi_y = J_{\xi_x,\xi_z}(\tfrac{T}{2}),
     	 \xi_z \in T_z \mathcal M     	 	 
     	 \right\},
\end{equation}
where $J_{\xi_x,\xi_z}$ is the Jacobi field along the unit speed geodesic 
$\gamma_{\overset{\frown}{x,z}}$ 
determined by $J(0) = \xi_x$ and $J(T)=\xi_z$.
Then the subdifferential $\partial \mathrm{d}_2$  at $(x,y,z) \in {\mathcal M}^3$ reads
\begin{equation}\label{eq:Subdiffd2}
     \partial \mathrm{d }_2 (x,y,z)   =  
     \left\{ \alpha \eta: \eta \in \mathcal V^\perp, \alpha \in [l_\eta,u_\eta] \right\},
\end{equation}     
where   
$
\mathcal V^\perp
$ 
denotes the set of normalized vectors
$\eta \coloneqq (\eta_x,\eta_y,\eta_z) \in T_x{\mathcal M} \times T_y{\mathcal M} \times T_z{\mathcal M}$
fulfilling
\[
\langle \xi_x,\eta_x \rangle_x + 
\langle \xi_y,\eta_y \rangle_y +
\langle \xi_z,\eta_z \rangle_z  = 0 
\]
for all $(\xi_x,\xi_y,\xi_z) \;  \in {\mathcal V}$,
and the interval endpoints $l_\eta$, $u_\eta$  are given by 
\begin{align}\label{rem:subdiffAux2} 
  l_\eta &= \lim_{\tau \uparrow 0} 
  \frac{\mathrm{d}_2(\exp_{x}( \tau \eta_x),\exp_{y}(\tau  \eta_y),\exp_{ z}(\tau \eta_z))}{\tau},\\
  u_\eta&= \lim_{\tau \downarrow 0} 
  \frac{\mathrm{d}_2(\exp_{x}( \tau \eta_x),\exp_{y}(\tau \eta_y),\exp_{z}(\tau \eta_z))}{\tau}. 
\end{align}
This can be seen as follows: For arbitrary tangent vectors $\xi_x$, $\xi_z$ sitting in $x$, $z$, respectively, 
we consider the uniquely determined geodesic variation 
$\Gamma (s,t)$ given by the side conditions $\Gamma (0,t) = \gamma_{\overset{\frown}{x,z}} (t),$
$\Gamma (s,0) = \exp_{x} \left(s \xi_x\right)$ 
as well as $\Gamma (s,T) = \exp_{z} \left(s \xi_z\right).$
We note that $c(\Gamma (s,0), \Gamma (s,T)) = \Gamma (s,T/2)$ which implies for $s \in (-\varepsilon, \varepsilon)$ that 
\begin{equation}\label{eq:rIsZero}
\mathrm{d}_2(\Gamma (s,0), \Gamma (s,T/2), \Gamma (s,T))  =   0.     
\end{equation} 
In view of the definition of a subgradient, see \cite{FO98,GH2013}, it is required, for a candidate $\eta$, that  
\begin{equation}\label{eq:defSubEquation}
\mathrm{d}_2(\exp_{x}( h_x),\exp_{y}( h_y),\exp_{z}( h_z)) - \mathrm{ d}_2(x,y,z) \geq 
\langle h,\eta\rangle + o(h).	
\end{equation}
for any sufficiently small $h\coloneqq(h_x,h_y,h_z)$.
Setting $h\coloneqq(\xi_x, J_{\xi_x,\xi_z}(T/2),\xi_z),$ equation~\eqref{eq:rIsZero} tells us
that the left hand side above equals $0$ up to $o(h)$, and thus, for a candidate $\eta$,
\[
 \langle \xi_x,\eta_x \rangle_x + 
\langle J_{\xi_x,\xi_z}(\tfrac{T}{2}),\eta_y \rangle_y +
\langle \xi_z,\eta_z \rangle_z 
=  o(h)	
\]
for all $\xi_x,\xi_y$ of small magnitude.
Since these are actually linear equations for $w=(\eta_x,\eta_y,\eta_z)$ in the tangent space, we get 
\[
\langle \xi_x,\eta_x \rangle_x + 
\langle J_{\xi_x,\xi_z}(\tfrac{T}{2}),\eta_y \rangle_y +
\langle \xi_z,\eta_z \rangle_z 
= 0.	
\]
Then, if $\eta \in \mathcal V^\perp$, the nominators in~\eqref{rem:subdiffAux2}
are nonzero and the limits exist.
Finally, we conclude~\eqref{eq:Subdiffd2} from~\eqref{eq:defSubEquation}.
\end{remark}
%-------------------------------------------------------------------------------

In the following subsection we recall how Jacobi fields can be computed for general symmetric spaces
and have a look at two special examples, namely
$n$-spheres and  manifolds of symmetric positive definite matrices.
%
%%------------------------------------------------------------------------------
\subsection{Jacobi Equation for Symmetric Spaces} \label{jacobi_symm_space}
%-------------------------------------------------------------------------------
Due to their rich structure symmetric spaces have been the object of differential geometric studies for a long time, 
and we refer 
to \cite{eschenburg1997lecture,eschenburg2014smmetric} or the books 
\cite{berger2003panoramic,jeff1975comparison} for more information.

A Riemannian manifold $\mathcal M$ is called {\em locally symmetric} if the geodesic reflection $s_x$ 
at each point $x \in \mathcal M$ given by mapping $\gamma(t) \mapsto \gamma(-t)$ 
for all geodesics $\gamma$ through $x = \gamma(0)$ is a local isometry, 
i.e., an isometry at least locally near $x$.
If this property holds globally, $\mathcal M$ it is called a {\em  (Riemannian globally) symmetric space}.
More formally, $\mathcal M$ is a  symmetric space if 
for any $x \in \mathcal M$ and  all $\xi \in T_x\mathcal S$ there is an isometry $s_x$ on  $\mathcal M$ such that 
$s_x (x) = x$ and $D_x s_x[\xi] = - \xi$.
A Riemannian manifold $\mathcal M$ is locally symmetric if and only if there exists a symmetric space which is locally 
isometric to $\mathcal M$. 
As a consequence of the Cartan–Ambrose–Hicks theorem \cite[Theorem 1.36]{jeff1975comparison},
every simply connected, complete, locally symmetric space is symmetric.
Symmetric spaces are precisely the homogeneous spaces with a symmetry $s_x$ at some point $x \in \mathcal M$.
Beyond $n$-spheres and  the manifold of symmetric positive definite matrices, 
hyperbolic spaces, Grassmannians as well as compact Lie groups are examples of symmetric spaces.
The crucial property we need is that a Riemannian manifolds is locally symmetric 
if and only if the covariant derivative of the Riemannian curvature tensor $R$ along curves is zero, i.e.,
\begin{equation}\label{connection}
 \nabla R = 0.
\end{equation}

%-------------------------------------------------------------------------------
\begin{proposition}\label{lem:basicStatementSym}
 Let $\mathcal M$ be a symmetric space. 
 Let $\gamma\colon [0,T] \rightarrow \mathcal M$  be a unit speed geodesic
 and $\{\Theta_1 = \Theta_1(t), \ldots, \Theta_n = \Theta_n(t)\}$  a parallel transported orthonormal frame 
 along $\gamma$. Let $J(t) = \sum_{i=1}^n a_i(t) \Theta_i(t)$ 
 be a Jacobi field of a variation through $\gamma$. Set $a \coloneqq (a_1.\ldots,a_n)^\tT$.
 Then the following relations hold true:
 \begin{enumerate}
 \item\label{constcoeffs} The Jacobi equation~\eqref{jacobi_diff} can be written as
\begin{equation} \label{eq:lemSymJacDglInFrame1}
	 a'' (t) +   G \,   a(t) = 0 ,
\end{equation}
 with the constant coefficient  matrix
 $G \coloneqq \left( \langle R( \Theta_i, \dot \gamma) \dot \gamma,  \Theta_j \rangle_{\gamma} \right)_{i,j=1}^n$.
\item
Let $\{ \theta_1 , \ldots, \theta_n \}$  be chosen as the initial orthonormal basis which
diagonalizes the operator 
\begin{equation} \label{diag} 
 \Theta \mapsto  R(\Theta, \dot \gamma) \dot \gamma
\end{equation}
at $t=0$ 
with corresponding eigenvalues $\kappa_i$, $i=1,\ldots,n$,
and let
$\{\Theta_1, \ldots, \Theta_n\}$ be the corresponding  parallel transported frame along $\gamma$.
Then the matrix $G$ becomes diagonal and 
\eqref{eq:lemSymJacDglInFrame1} decomposes into the $n$  ordinary linear  differential equations  
	\begin{equation}\label{eq:lemSymJacDglInFrame3}
	a_i'' (t) +  \kappa_i   a_i(t) = 0,   \qquad i= 1,\ldots,n.
	\end{equation}
\item
	The Jacobi fields
\begin{equation}  \label{eq:trafoJacobl0} 
	 J_k (t) \coloneqq
	\begin{cases}
	\sinh (\sqrt{-\kappa_k} t) \ \Theta_k(t) ,  
	&\quad \mathrm{ if } \; \kappa_k < 0, \\
	\sin (\sqrt{\kappa_k} t) \ \Theta_k(t)  ,  
	&\quad \mathrm{ if } \; \kappa_k > 0, \\
	t \ \Theta_k(t),
	&\quad \mathrm{ if } \; \kappa_k = 0,
	\end{cases}
	\end{equation}
	$k=1,\ldots,n$ form a basis of the $n$ dimensional linear space of Jacobi fields of a variation through $\gamma$
	fulfilling the initial condition $J(0) = 0$.
\end{enumerate}
\end{proposition}

%-------------------------------------------------------------------------------
Part~\ref{constcoeffs}) of Proposition~\ref{lem:basicStatementSym} is  also stated as Property A in Rauch's paper \cite{Rau61} in the 
particularly nice form ``The curvature of a 2-section propagated parallel along a geodesic is constant.''  
For convenience we add the proof.

\begin{proof} i) Using the frame representation of $J$, \eqref{trans} and the linearity of $R$ in the first argument, 
the  Jacobi equation~\eqref{jacobi_diff} becomes
\begin{align}
 0 &= \sum_{i=1}^n a_i''(t) \Theta_i(t) + a_i(t) R(\Theta_i,\dot \gamma) \dot \gamma,
 \end{align}
 and by taking inner products with $\Theta_j$ further
 \begin{align}
 0 &= a_j''(t) + \langle R( \Theta_i, \dot \gamma) \dot \gamma,  \Theta_j \rangle a_i(t), \quad j = 1,\ldots,n.
\end{align}
Now~\eqref{connection} implies for $R( \Theta_i, \dot \gamma) \dot \gamma = \sum_{k=1}^n r_{ik}(t) \Theta_k(t)$
\[
0 = \nabla_{\dot \gamma} R = \sum_{k=1}^n r_{ik}'(t) \Theta_k(t)
\]
which is, by the linear independence of the $\Theta_k$, $k=1,\ldots,n$, only possible if all $r_{ik}$ are constants
and we get $G = \left(r_{ij} \right)_{i,j=1}^n$.

Parts ii) and iii) follow directly from i). For iii) we also refer to
\cite[p. 77]{jeff1975comparison}.
\end{proof}

With respect to our special Jacobi fields in Lemma~\ref{jacobi} we obtain the following corollary.

%
%-------------------------------------------------------------------------------
\begin{corollary} \label{dreher}
 The Jacobi fields $J_k$, $k=1,\ldots,n$ of a variation through $\gamma_{\overset{\frown}{x,z}} = \gamma_{x,\upsilon}$ 
 with boundary conditions 
 $J_k(0) = \xi_k$  and $J_k(T) = 0$
 fulfill 
 \begin{equation}  \label{eq:trafoJacobl2}  
	J_k (\tfrac{T}{2})=         
	\begin{cases}
	\frac{\sinh \left(\sqrt{-\kappa_k} \tfrac{T}{2} \right)}{\sinh (\sqrt{-\kappa_k} T)} \Xi_k (\tfrac{T}{2}) ,  
	&\quad \mathrm{ if }\ \kappa_k < 0, \\[1ex]
	\frac{\sin \left(\sqrt{\kappa_k} \tfrac{T}{2} \right)}{\sin (\sqrt{\kappa_k} T)}  \Xi_k (\tfrac{T}{2}),  
	&\quad \mathrm{ if }\ \kappa_k > 0, \\[1ex]
	\frac{1}{2}    \Xi_k (\tfrac{T}{2}),
	&\quad \mathrm{ if }\  \kappa_k = 0.
	\end{cases}
	\end{equation}
\end{corollary}
%-------------------------------------------------------------------------------

%
\begin{proof}
 The Jacobi fields $\bar J_k (t) \coloneqq \alpha_k J_k(T-t)$ 
 of a variation trough $\gamma_{\overset{\frown}{z,x}} \coloneqq \gamma_{x,\upsilon}(T-t)$ satisfy $J_k(0) = 0$ 
 and by Proposition~\ref{lem:basicStatementSym} iii) they are given by
 \[
	 \bar J_k (t) \coloneqq       
	\begin{cases}
	\sinh (\sqrt{-\kappa_k} t) \ \Xi_k(T-t) ,  
	&\quad \mathrm{ if } \; \kappa_k < 0, \\
	\sin (\sqrt{\kappa_k} t) \ \Xi_k(T-t)  ,  
	&\quad \mathrm{ if } \; \kappa_k > 0, \\
	t \ \Xi_k(T-t),
	&\quad \mathrm{  if } \; \kappa_k = 0.
	\end{cases}
\]
In particular we have
\[
	 \bar J_k (T) \coloneqq       
	\begin{cases}
	\sinh (\sqrt{-\kappa_k} T) \ \xi_k ,  
	&\quad\mathrm{ if } \; \kappa_k < 0, \\
	\sin (\sqrt{\kappa_k} T) \ \xi_k  ,  
	&\quad \mathrm{ if } \; \kappa_k > 0, \\
	T \ \xi_k,
	&\quad \mathrm{ if } \; \kappa_k = 0.
	\end{cases}
\]
Now $ \alpha_k \xi_k = \alpha_k J_k(0) = \bar J_k(T)$ determines $\alpha_k$ as
\begin{equation*}
	\alpha_k = 
	\begin{cases}
	\sinh (\sqrt{-\kappa_k} T),  
	&\quad \mathrm{ if }\;  \kappa_k < 0, \\
	\sin (\sqrt{\kappa_k} T),  
	&\quad \mathrm{ if }\;  \kappa_k > 0, \\
	T,
	&\quad \mathrm{ if }\;  \kappa_k = 0
	\end{cases}
\end{equation*}
and $J_k (t)=  \tfrac{1}{\alpha_k} \bar  J_k(T-t)$.
We notice that the denominators of the appearing fractions are nonzero since $x$ and $z$
were assumed to be non-conjugate points. Finally, we get~\eqref{eq:trafoJacobl2}.  
\end{proof}

Let us apply our findings for the $n$-sphere and the manifold of symmetric positive definite matrices.
%
%-------------------------------------------------------------------------------
\paragraph{The Sphere \texorpdfstring{${\mathbb S}^n$}{\texorpdfstring{${\mathbb S}^2$}{S\textsuperscript{n}}}.}
%-------------------------------------------------------------------------------
We consider the $n$-sphere $\mathbb S^n$.
Then, the  Riemannian metric is just the Euclidean distance $\langle \cdot, \cdot \rangle_x = \langle \cdot, \cdot \rangle$ in $\mathbb R^{n+1}$. 
For the definitions of the  geodesic distance, the  exponential map and parallel
transport see Appendix~\ref{A-sphere}. 
Let $\gamma \coloneqq \gamma_{x,\upsilon}$.
We choose $\xi_1 \coloneqq \upsilon = \dot\gamma(0)$ and complete this to an orthonormal basis 
$\{\xi_1, \ldots,\xi_n\}$ of $T_x\mathbb S^n$ with corresponding parallel frame
$\{\Xi_1, \ldots,\Xi_n\}$ along $\gamma$.
Then diagonalizing the operator~\eqref{diag} is especially simple.
Since $\mathbb S^n$ has constant curvature $C=1$, the Riemannian curvature tensor fulfills \cite[Lemma 8.10]{Lee97}~$
R(\Theta,\Xi) \Upsilon = \langle \Xi,\Upsilon\rangle \Theta -  \langle \Theta,\Upsilon\rangle \Xi$. Consequently,
\[
R(\Theta,\dot \gamma) \dot \gamma = \langle \dot \gamma,\dot \gamma \rangle \Theta -  \langle \Theta,\dot \gamma\rangle \dot \gamma
= \Theta -  \langle \Theta,\dot \gamma\rangle \dot \gamma
\]
so that at $t=0$ the vector $\xi_1$ is an eigenvector with eigenvalue $\kappa_1 =0$ and $\xi_i$, $i=2,\ldots,n$, are eigenvectors
with eigenvalues $\kappa_i = 1$. Consequently, we obtain by Lemma~\ref{jacobi} and Corollary~\ref{dreher} the following corollary.
%-------------------------------------------------------------------------------
\begin{corollary} \label{jacobi_sphere}
For the sphere $\mathbb S^n$ and the above choice of the orthonormal frame system, the following relations hold true:
\[
D_x c [\xi_1] = \frac{1}{2} \, \Xi_1 (\tfrac{T}{2}),
\quad
D_x c [\xi_k] = \frac{\sin \frac{T}{2}}{\sin T} \, \Xi_k(\tfrac{T}{2}) , \quad k=2,\ldots,n.
\]
\end{corollary}
%
%-------------------------------------------------------------------------------
\paragraph{Symmetric Positive Definite Matrices.}
%-------------------------------------------------------------------------------
Let $\text{Sym} (r)$ denote the space of symmetric $r \times r$ matrices with (Frobenius) inner product and norm
\begin{equation} \label{inner_frob}
\langle A,B \rangle \coloneqq \sum_{i,j=1}^r a_{ij}b_{ij}, \quad \|A\| \coloneqq \left( \sum_{i,j=1} a_{ij}^2 \right)^{\frac12}.
\end{equation}
Let $\mathcal P(r)$ be the manifold of symmetric positive definite $r \times r$ matrices.
It has the dimension $\text{dim} \, {\mathcal P}(r) = n = \frac{r(r+1)}{2}$.
The tangent space of ${\mathcal P}(r)$ at $x \in {\mathcal P}(r)$ 
is given by 
$T_x {\mathcal P}(r) = \{x\} \times \text{Sym}(r) = \{ x^{\frac12} \eta x^{\frac12} : \eta \in \text{Sym} (r) \}$,
in particular $T_I {\mathcal P} (r) = \text{Sym} (r)$, where $I$ denotes the $r \times r$ identity matrix.
The  Riemannian metric on $T_x {\mathcal P}$ reads
\begin{equation} \label{rm_spd}
\langle \eta_1,\eta_2 \rangle_x \coloneqq
\text{tr} (\eta_1 x^{-1} \eta_2 x^{-1}) =
\langle x^{-\frac12} \eta_1 x^{-\frac12}, x^{-\frac12} \eta_2 x^{-\frac12}\rangle,
\end{equation}
where $\langle \cdot,\cdot \rangle$ denotes the matrix inner product~\eqref{inner_frob}.
For the definitions of the geodesic distance, exponential map, parallel
transport see the Appendix~\ref{A-spd}.

Let $\gamma \coloneqq \gamma_{x,\upsilon}$ and let the matrix $v \in T_x{\mathcal P}$ 
have the eigenvalues $\lambda_1,\ldots,\lambda_r$
with a corresponding orthonormal basis  of eigenvectors $v_1,\ldots,v_r$ in $\mathbb R^r$, i.e.,
\begin{equation} \label{ev_dec}
v = \sum_{i=1}^r \lambda_i v_i v_i^\tT. 
\end{equation}
We will use a more appropriate index system for the frame~\eqref{frames_par}, namely
\[
{\mathcal I} \coloneqq \{(i,j)\,:\, i=1,\ldots,r; j = i,\ldots,r\}.
\]
Then the matrices
\begin{equation} \label{basis}
\xi_{ij} \coloneqq \begin{cases}
\frac12(v_i v_j^\tT + v_j v_i^\tT), \quad (i,j) \in {\mathcal I} & \mathrm{ if } \; i =j,	
\\
\frac{1}{\sqrt{2}} (v_i v_j^\tT + v_j v_i^\tT), \quad (i,j) \in {\mathcal I} & \mathrm{ if } \; i\not =j,
\end{cases}
\end{equation}
form an orthonormal basis of $T_x {\mathcal P}(r)$. 

In other words, we will deal with the parallel transported frame $\Xi_{ij}$, $(i,j) \in {\mathcal I}$, of~\eqref{basis} instead of $\Xi_k$, $k=1,\ldots,n$.
To diagonalize the operator~\eqref{diag} at $t=0$ we use that
the  Riemannian curvature tensor for ${\mathcal P} (r)$ has the form
\begin{equation}\label{A_spd_R}
R(\Theta,\Xi) \Upsilon = - \tfrac14 x^{\frac12} \left[ [x^{-\frac12} \Theta x^{-\frac12}, x^{-\frac12} \Xi x^{-\frac12}] , x^{-\frac12} \Upsilon x^{-\frac12}\right] x^{\frac12}
\end{equation}
with the Lie bracket $[A,B] = AB - BA$ of matrices. Then
\[
R(\Theta,\dot \gamma) \dot \gamma = - \tfrac14 x^{\frac12} \left[ [x^{-\frac12} \Theta x^{-\frac12}, x^{-\frac12} \dot \gamma x^{-\frac12}] , x^{-\frac12} \dot \gamma x^{-\frac12}\right] x^{\frac12}
\]
and for $t=0$ with $\theta = \Theta(0)$ the right-hand side becomes 
\begin{align}
T(\theta) 
&=  - \tfrac14 x^{\frac12} \left[ [x^{-\frac12} \theta x^{-\frac12}, x^{-\frac12} \upsilon x^{-\frac12}] , x^{-\frac12} \upsilon x^{-\frac12}\right] x^{\frac12}\\
&=  - \tfrac14 x^{\frac12} (w b^2 - 2 b w b +b^2w) x^{\frac12},
\end{align}
where $b\coloneqq x^{-\frac12}\upsilon x^{-\frac12}$ and $w \coloneqq x^{-\frac12} \theta x^{-\frac12}$.
Expanding $\theta = \sum_{(i,j) \in {\mathcal I}} \mu_{ij} \xi_{ij}$ into the  orthonormal basis of $T_x{\mathcal P}$
and substituting this into $T(\theta)$ gives after a straightforward computation
\[
T(\theta) = \sum_{(i,j) \in {\mathcal I}}  \mu_{ij} (\lambda_i -\lambda_j)^2 \xi_{ij}.
\]
Thus $\{\xi_{ij}: (i,j) \in {\mathcal I} \}$ is an orthonormal basis of eigenvectors of $T$ with corresponding eigenvalues
\[
	\kappa_{ij} = - \tfrac14 (\lambda_i - \lambda_j)^2, \qquad (i,j) \in {\mathcal I}.
\]
Let 
$
{\mathcal I}_1 
\coloneqq \{(i,j) \in
\mathcal I
: \lambda_i = \lambda_j\}
$ and $
{\mathcal I}_2
\coloneqq \{(i,j) \in
\mathcal I: \lambda_i \not = \lambda_j\}.
$
Then, by Lemma~\ref{jacobi} and Corollary~\ref{dreher}, we get the following corollary.
\begin{corollary} \label{cor:spd}
 For the manifold of symmetric positive definite matrices ${\mathcal P}(r)$ it holds
\[
 D_x c [\xi_{ij}] = 
\begin{cases}
	\frac12 \Xi_{ij}(\tfrac{T}{2})
 & \mathrm{ if } \; (i,j) \in {\mathcal I}_1,\\[1ex]
    \frac{\sinh\left(\frac{T}{4} 
|\lambda_i - \lambda_j|\right) }{\sinh\left(\frac{T}{2} |\lambda_i - \lambda_j| \right)} \, \Xi_{ij}(\tfrac{T}{2})      
&\mathrm{ if } \; (i,j) \in {\mathcal I}_2.
\end{cases}
\]
\end{corollary}
%
%-------------------------------------------------------------------------------
\section{Inexact Cyclic Proximal Point Algorithm} \label{sec:cppa}
%-------------------------------------------------------------------------------
In order to minimize the functional in~\eqref{task_2}, 
we follow the approach in~\cite{BLSW14} and 
employ a cyclic proximal point algorithm (cyclic PPA). 

For a proper, closed, convex function $\phi\colon \mathbb R^m \rightarrow (-\infty,+\infty]$
and $\lambda > 0$ the \emph{proximal mapping} $\prox_{\lambda \phi}\colon\mathbb R^m \rightarrow \mathbb R^m$
at $x \in \mathbb R^m$ is defined by
\begin{equation} \label{prox_R}
	\prox_{\lambda \phi}(x)
	\coloneqq \argmin_{y \in \mathbb R^m} \bigg\{ \frac{1}{2\lambda} \lVert x-y\rVert_2^2 + \phi (y) \bigg\},
\end{equation}
see~\cite{Mor62}.
The above minimizer exits and is uniquely determined.
Many algorithms which were recently used in variational image processing reduce
to the iterative computation of values of proximal mappings.
An overview of applications of proximal mappings is given in~\cite{PB2013}.

Proximal mappings were generalized for functions on Riemannian manifolds in~\cite{FO02}, 
replacing the squared Euclidean norm by the squared geodesic distances. 
For $\phi\colon {\mathcal M}^m \rightarrow  (-\infty,+\infty]$
and $\lambda > 0$ let
\begin{equation} \label{prox_mani}
\prox_{\lambda \phi }(x)
\coloneqq \argmin_{y \in {\mathcal M}^m } \biggl\{ \frac{1}{2 \lambda} \sum_{j=1}^m d_{\mathcal M} (x_j,y_j) ^2
	+ \phi(y)\biggr\}.
\end{equation}
For proper, closed, convex functions $\phi$ on Hadamard manifolds 
the minimizer exits and is uniquely determined.
More generally, one can define proximal mappings in certain metric spaces. 
In particular, such a definition was given independently in~\cite{Jost95}
and~\cite{Mayer} for Hadamard spaces,
which was later on used for the PPA~\cite{Bac13} and cyclic PPA~\cite{Bac14}.
%
%-------------------------------------------------------------------------------
\subsection{Algorithm} \label{sec:alg}
%-------------------------------------------------------------------------------
We split the functional in~\eqref{task_2} into the summands
\begin{equation} \label{split}
{\mathcal E} = \sum_{l = 1}^{15} {\mathcal E}_l,
\end{equation}
where ${\mathcal E}_1(u) \coloneqq F(u; f)$ and
\begin{align} \label{eq:splitTVevodd}
	\alpha \operatorname{TV}_1(u)
	&=	
\sum_{\nu_1=0}^1
	\alpha_1 \sum_{i,j=1}^{\bigl\lfloor\!\frac{N-1}{2}\!\bigr\rfloor,M}
		 d_{\mathcal M} (u_{2i-1+\nu_1,j},u_{2i+\nu_1,j})
		\\
		&\qquad+
\sum_{\nu_2=0}^1	\alpha_2 \sum_{i,j=1}^{N,\bigl\lfloor\!\frac{M-1}{2}\!\bigr\rfloor}
		 d_{\mathcal M} (u_{i,2j-1+\nu_2},u_{i,2j+\nu_2})
		\\
	&=: \sum_{\nu_1=0}^1 {\mathcal E}_{2+\nu_1}(u) + \sum_{\nu_2=0}^1 {\mathcal E}_{4+\nu_2}(u) 
\end{align}
and 
\begin{align*}
\beta&\operatorname{TV}_2(u)\\
&=
\sum_{\nu_1=0}^2
		\beta_1\sum_{i,j=1}^{\bigl\lfloor\!\frac{N-1}{3}\!\bigr\rfloor,M}
			\mathrm{d}_2(u_{3i-2+\nu_1,j},u_{3i-1+\nu_1,j},u_{3i+\nu_1})\\
&\quad+
	\sum_{\nu_2=0}^2
		\beta_2\sum_{i,j=1}^{N,\bigl\lfloor\!\frac{M-1}{3}\!\bigr\rfloor}
			\mathrm{d}_2(u_{i,3j-2+\nu_2},u_{i,3j-1+\nu_2},u_{i,3j+\nu_2})
			\\
&\quad+\!\!\!\!
	\sum_{\nu_3,\nu_4=0}^1
		\!\!\!\beta_3\!\!\!\!\!\!\sum_{i,j=1}^{\bigl\lfloor\!\frac{N-1}{2}\!\bigr\rfloor,\bigl\lfloor\!\frac{M-1}{2}\!\bigr\rfloor}\!\!\!\!\!\!\!\!
			\mathrm{d}_{1,1}
			(
			{u_{2i-1+\nu_3,2j-1+\nu_4},u_{2i+\nu_3,2j-1+\nu_4},u_{2i-1+\nu_3,2j+\nu_4},u_{2i+\nu_3,2j+\nu_4}})
			\\
	& =: \sum_{\nu_1=0}^2 \mathcal E_{6+\nu_1}(u) + \sum_{\nu_1=0}^2 \mathcal E_{9+\nu_1}(u) + \sum_{\nu_3,\nu_4=0}^1 \mathcal E_{12+\nu_3 + 2\nu_4}(u).
\end{align*}
Then the exact cyclic PPA computes starting with $u^{(0)} = f$ until a convergence criterion is reached
the values
\begin{equation} \label{cppa_form}
u^{(k+1)} \coloneqq \prox_{\lambda_k {\mathcal E}_{15}} \circ \prox_{\lambda_k {\mathcal E}_{14}} \circ \ldots\circ 
\prox_{\lambda_k {\mathcal E}_1}  (u^{(k)}) 
\end{equation}
where  the parameters $\lambda_k>0$ in the $k$-th cycle have to fulfill
\begin{equation}\label{eq:CPPAlambda}
		\sum_{k=0}^\infty \lambda_k = \infty,
		\quad \text{and}
		\quad \sum_{k=0}^\infty \lambda_k^2 < \infty.
\end{equation}
By construction, the functional \({\mathcal E}_l\), \(l\in\{1,\ldots,15\}\) in \eqref{cppa_form}, contains every entry of \(u\) at most once. Hence the involved proximal mappings of \(\prox_{\lambda\mathcal E_l}\) consists of can be evaluated by computing all involved proximal mappings, one for every summand, in parallel, i.e. for
\begin{itemize}
\item[(D0)]  $d_{\mathcal M}^2(u_{ij},f_{ij})$ of the data fidelity term,
\item[(D1)]  $\alpha_1 d_{\mathcal M} (u_{2i-1+\nu_1,j}, u_{2i-\nu_1,j})$,
$\alpha_2 d_{\mathcal M} (u_{i,2j-1+\nu_2},u_{i,2j-\nu_2})$ 
of the first order differences,
\item[(D2)] $\beta_1 d_2(u_{3i-2+\nu_1,j},u_{3i-1+\nu_1,j},u_{3i+\nu_1})$, 
$\beta_2 d_2(u_{i,3j-2+\nu_2},u_{i,3j-1+\nu_2},u_{i,3j+\nu_2})$ of the second order differences, 
and $\beta_3 \mathrm{d}_{1,1}(
u_{2i-1+\nu_3,2j-1+\nu_4},u_{2i+\nu_3,2j-1+\nu_4},$\\$u_{2i-1+\nu_3,2j+\nu_4},u_{2i+\nu_3,2j+\nu_4}
%...
)$ 
of the second order mixed differences.
\end{itemize}
Taking these as the functions $\phi$ which are of interest in~\eqref{prox_mani} we can reduce our attention
to $m=1,2$ and $m=3,4$, respectively.
Analytical expressions for the minimizers defining the proximal mappings, 
for the data fidelity terms (D0) are given in~\cite{FO02}, 
and for the first order differences (D1) in~\cite{WDS2014}.
For the second order difference in (D2) such
expressions are only available for the manifold${\mathcal M} = \mathbb S^1$,
see~\cite{BLSW14}.

\begin{figure}\centering
		\begin{subfigure}{.45\textwidth}\centering
			\includegraphics[width=.66\textwidth]{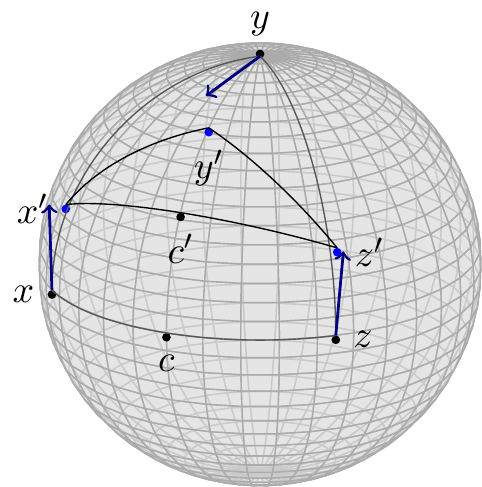}
			\caption{An inexact \(\prox_{\lambda\mathrm{d}_2}\), \(\lambda=2\pi\).}\label{subfig:2DiffProx1}
	\end{subfigure}
		\begin{subfigure}{.45\textwidth}\centering
			\includegraphics[width=.66\textwidth]{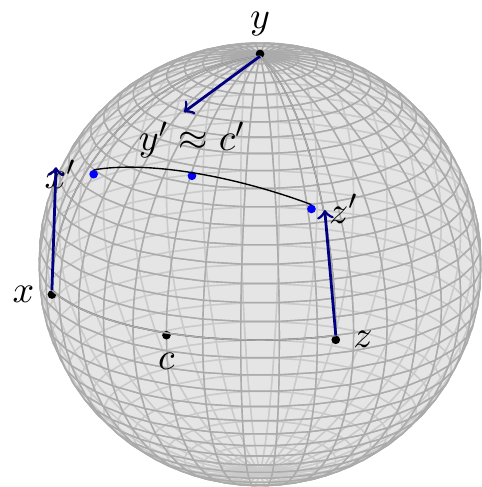}
			\caption{An inexact \(\prox_{\lambda\mathrm{d}_2}\), \(\lambda=4\pi\).}\label{subfig:2DiffProx2}
	\end{subfigure}
	\caption{Illustration of the inexact proximal mapping \((x',y',z') = \prox_{\lambda\mathrm{d}_2}(x,y,z)\). 
	For all points the negative gradients are shown in dark blue.}\label{fig:2DiffProx}
\end{figure} 

In order to derive an approximate solution of
\[
\prox_{\lambda \mathrm{d}_2} (g_1,g_2,g_3)
=
\argmin_{x \in {\mathcal M}^3} 
\biggl\{
	\frac{1}{2} \sum_{j=1}^3 d_{\mathcal M} (x_j,g_j) ^2
	+ \lambda \mathrm{d}_2 (x_1,x_2,x_3)
	\biggr\}
=\vcentcolon \argmin_{x \in {\mathcal M}^3} \psi(x)
\]
we employ the (sub)gradient descent method to $\psi$. 
For gradient descent methods on manifolds including convergence results we
refer to~\cite{AMS08,Udriste94}.
The subgradient method is one of the classical algorithms for nondifferentiable optimization 
which was extended for manifolds, e.g., in~\cite{FO98,GH2013}. 
In~\cite{FO98} convergence results for Hadamard manifolds were established. 
A subgradient method on manifolds is given in Algorithm~\ref{alg:Subgrad}. 
In particular, again restricting to the second order differences, we have to compute the gradient of $\psi$:
\[
\nabla_{\mathcal M^3} \psi(x)
= 
-\begin{pmatrix}
\log_{x_1} g_1\\
\log_{x_2} g_2\\
\log_{x_3} g_3\\
\end{pmatrix}
+ \lambda \nabla_{\mathcal M^3}\mathrm{d}_2(x_1,x_2,x_2), \quad c(x_1,x_3) \not = x_2.
\]
The computation of $\nabla_{\mathcal M^3}\mathrm{d}_2$ was the topic of Section~\ref{sec:order_two}. A result of Algorithm~\ref{alg:Subgrad} 
for the points already used in Fig.~\ref{subfig:2DiffS}, the Fig.~\ref{fig:2DiffProx} illustrates the proximal mapping for two different values of \(\lambda\).

In summary this means that we perform an inexact cyclic PPA as in Algorithm~\ref{alg:CPPA}.
We will prove the convergence of such an algorithm in the following subsection for Hadamard spaces.
%
%-------------------------------------------------------------------------------
\begin{algorithm}[tbp]
	\caption[]{\label{alg:Subgrad} Subgradient Method for \(\prox_{\lambda\mathrm{d}_2}\)}
	\begin{algorithmic}
		\State \textbf{Input} data $g=(g_1,g_2,g_3)\in\mathcal M^3$, a sequence $\tau=\{ \tau_k \}_k\in\ell_2\backslash\ell_1$.
		\\\vspace{-.5\baselineskip}
		\Function {SubgradientProxD2}{$g$, $\tau$}
		\State Initialize \(x^{(0)}=g\), \(x^{*} = x^{(0)}\), \(k=1\).
		\Repeat
			\State $x^{(k)} \gets \exp_{x^{(k-1)}}\Bigl(-\tau_k\nabla_{\mathcal M^3} \psi(x^{(k-1)})\Bigr)$
			\If{$\psi(x^{*}; g) > \psi(x^{(k)}; g)$}
				$x^{*} \gets x^{(k)}$
			\EndIf
			\State $k\gets k+1$
		\Until a convergence criterion is reached
		\State\Return $x^{*}$
		\EndFunction
	\end{algorithmic}
\end{algorithm}
%-------------------------------------------------------------------------------
%
%-------------------------------------------------------------------------------
\begin{algorithm}[tbp]
	\caption[]{\label{alg:CPPA} Inexact Cyclic PPA for minimizing~\eqref{task_2}}
	\begin{algorithmic}
		\State \textbf{Input} data $f\in\mathcal M^{N\times M}$, $\alpha \in\mathbb R_{\geq 0}^2$, $\beta \in\mathbb R_{\geq 0}^3$,
		a sequence $\lambda = \{\lambda_k \}_k$, $\lambda_k>0$, fulfilling~\eqref{eq:CPPAlambda},
		and a sequence of positive reals $\epsilon=\{\epsilon_k\}_k$ with $\sum \epsilon_k<\infty$
		\Function {CPPA}{$\alpha$, $\beta$, $\lambda$, $f$}
		\State Initialize \(u^{(0)}=f\), \(k=0\)
		\State Initialize the cycle length as \(L = 15\) (or \(L=6\) for the case \(M=1\) ).
		\Repeat
		\For{$l \gets 1$ \textbf{to} $L$}
			\State 
			\( u^{(k+\frac{l}{L})} \leftarrow \prox_{\lambda_k {\varphi}_l}(u^{(k+\frac{l-1}{L})})\),
			\State where the proximal operators are given analytically for (D0)
			\State and (D1) as in \cite{FO02,WDS2014} and approximately for (D2) via Algorithm~\ref{alg:Subgrad}
			\State and the error is bounded by \(\epsilon_k\)
		\EndFor
		\State $k \gets k+1$
		\Until a convergence criterion is reached
		\State\Return $u^{(k)}$
		\EndFunction
	\end{algorithmic}
\end{algorithm}
%-------------------------------------------------------------------------------
\subsection{Convergence Analysis} \label{sec:conv}
%-------------------------------------------------------------------------------
We now present the convergence analysis of the above algorithms in the setting of Hadamard spaces, which include, for instance, the manifold of symmetric positive definite matrices. Recall that a complete metric space $(X,d)$ is called Hadamard if every two points $x,y$ are connected by a geodesic and the following condition holds true
\begin{equation}\label{eq:reshet}
d(x,v)^2 + d(y,w)^2 \le d(x,w)^2 + d(y,v)^2 + 2 d(x,y)d(v,w),
\end{equation}
for any $x,y,v,w \in X.$ Inequality~\eqref{eq:reshet} implies that Hadamard spaces have nonpositive curvature \cite{Alex51,Re68} and Hadamard spaces are thus a natural generalization of complete simply connected Riemannian manifolds of nonpositive sectional curvature. For more details, the reader is referred to \cite{mybook,Jost}.

In this subsection, let $({\mathcal H},d)$ be a locally compact Hadamard space. 
We consider 
\begin{equation} \label{eq:f}
 {\varphi} =\sum_{l=1}^L \varphi_l, 
\end{equation}
where $\varphi_l\colon {\mathcal H} \rightarrow \mathbb R$ are convex continuous functions and 
assume that $\varphi$ attains a (global) minimum.

For Hadamard spaces ${\mathcal H} \coloneqq {\mathcal M}^\text{ N}$, $\text{ N} = N \cdot M$,
the functional $\varphi = {\mathcal E}$ in~\eqref{split} 
fits into this setting with $L=15$. Alternatively we may take the single differences in (D0)-(D2)
as summands $\varphi_l$.
Our aim is to show the convergence of the (inexact) cyclic PPA.
To this end, recall that, given a metric space $(X,d)$,
a mapping $T\colon X\to X$ is nonexpansive if $d(Tx,Ty)\le d(x,y)$.
In the proof of Theorem~\ref{thm:cyclic}, we shall need the following well
known lemmas. Lemma \ref{lem:estim} is a consequence of the strong
convexity of a regularized convex function and expresses how much the
function's value decreases after applying a single PPA step.
Lemma~\ref{assumption_lem} is a refinement of the fact that a bounded
monotone sequence has a limit.

\begin{lemma}[{\cite[Lemma 2.2.23]{mybook}}] \label{lem:estim}
If $h\colon {\mathcal H} \rightarrow (-\infty,+\infty]$ is a convex lower semi-continuous function, 
then, for every $x,y \in {\mathcal H}$, we have
\begin{equation*} 
h \left( \prox_{\lambda h} (x) \right) - h(y) \leq \frac1{2\lambda} d(x,y)^2 -\frac1{2\lambda} d \left(\prox_{\lambda h} (x),y \right)^2.
\end{equation*}
\end{lemma}
\begin{lemma} \label{lem:detmartingale}
Let $\{a_k\}_{k\in \N}$, $\{b_k\}_{k\in \N}$, $\{ c_k \}_{k\in \N}$ and $\{ \eta_k\}_{k\in \N}$ 
be sequences of nonnegative real numbers. 
For each $k \in \N$ assume 
\begin{align} \label{ineq}
a_{k+1} & \leq \left( 1+\eta_k \right)a_k -b_k +c_k,
\end{align}
along with
\begin{equation} \label{assumption_lem}
\sum_{k=1}^\infty c_k  < \infty \quad \text{and} \quad \sum_{k=1}^\infty \eta_k  < \infty. 
\end{equation}
Then the sequence $\{a_k\}_{k \in \N}$ converges and $\sum_{k=1}^\infty b_k < \infty.$
\end{lemma}
%-------------------------------------------------------------------------------

Let us start with the exact cyclic PPA. The following theorem generalizes \cite[Theorem 3.4]{Bac14} in a way that is required for proving
convergence for our setting.
The point $p$ in Theorem 4.3 is a reference point chosen arbitrarily. In linear spaces it is natural to take the origin. Condition~\eqref{i:sppa:lips}
then determines how fast the functions $\varphi_l$ can change their values across the space.
%-------------------------------------------------------------------------------
\begin{theorem}[Cyclic PPA] \label{thm:cyclic} 
Let $({\mathcal H},d)$ be a locally compact Hadamard space and let $\varphi$ in~\eqref{eq:f}
have a global minimizer.
Assume that there exist $p \in {\mathcal H}$ and $C>0$ such that for each $l=1,\dots,L$ and all $x,y\in {\mathcal H}$ we have 
\begin{equation} \label{i:sppa:lips}
\varphi_l(x)-\varphi_l(y) \leq C d(x,y) \left(1+d(x,p) \right).
\end{equation}
Then the sequence $\{ x^{(k)} \}_{k \in \mathbb N}$ defined  by the cyclic PPA 
\begin{equation} \label{cppa_form2}
x^{(k+1)} \coloneqq \prox_{\lambda_k {\varphi}_{L}} \circ \prox_{\lambda_k {\varphi}_{L-1}} \circ 
\ldots \circ \prox_{\lambda_k {\varphi}_1}  (x^{(k)}) 
\end{equation}
with $\{\lambda_k\}_{k \in \mathbb N}$ as in~\eqref{eq:CPPAlambda} converges for every starting point $x^{(0)}$ 
to a minimizer of $\varphi$.
\end{theorem}
%-------------------------------------------------------------------------------
\begin{proof}
For $l=1,\ldots L$ we set 
\[x^{(k + \tfrac{l}{L})} \coloneqq \prox_{\lambda_k {\varphi}_l} (x^{( k + \tfrac{l-1}{L} )}).\]

\noindent 1. First we prove that for any fixed $q \in {\mathcal H}$ and all $k \in \mathbb N_0$ there exists a constant $C_q>0$ such that
\begin{equation} \label{eq:estimcyclic}
d\big( x^{(k+1)},q \big)^2 
\leq 
\big(1 + C_q\lambda_k^2 \big) d\big( x^{(k)},q \big)^2 -2\lambda_k \left(\varphi \big(x^{(k)} \big)- \varphi(q) \right) + C_q \lambda_k^2.
\end{equation}
For any  fixed $q \in {\mathcal H}$ we obtain by~\eqref{i:sppa:lips} and the triangle inequality
\begin{equation} \label{eq:growthagain} 
\varphi_l(x) - \varphi_l (y) \leq C_q d(x,y) \left(1+d(x,q) \right), \quad C_q \coloneqq 1 + d(q,p).
\end{equation}
Applying Lemma~\ref{lem:estim} with $h \coloneqq \varphi_l,\; x \coloneqq x^{(k+ \frac{l-1}{L})}$ and $y\coloneqq q$ we conclude
\begin{equation*} 
d \big( x^{(k + \tfrac{l}{L} ) },q \big)^2 
\leq 
d \big(x^{( k + \tfrac{l-1}{L} )},q \big)^2 
- 2 \lambda_k \left( \varphi_l \big( x^{( k + \tfrac{l}{L} )} ) \big) - \varphi_l(q) \right) 
\end{equation*}
for $l=1,\dots,L$. Summation yields
\begin{align}\label{eq:supermartingale1}
d \big(x^{(k+1)},q \big)^2 & \leq d \big(x^{(k)},q \big)^2 -
2\lambda_k\sum_{l=1}^L \Bigl( \varphi_l (x^{(k + \tfrac{l}{L})} ) - \varphi_l (q) \Bigr) 
\nonumber \\  
& = d \big( x^{(k)},q \big)^2 -2\lambda_k \bigl(\varphi \big(x^{(k)} \big) - \varphi(q) \bigr) + 
2\lambda_k \sum_{l=1}^L \Bigl(\!\varphi_l (x^{(k)} ) - \varphi_l \big( x^{(k + \tfrac{l}{L})} \big)\!\Bigr)\!,
\end{align}
where we used~\eqref{eq:f}. The growth condition in~\eqref{eq:growthagain} gives
\begin{equation} \label{eq:growthagain_1}
\varphi_l\big(x^{(k)}\big)-\varphi_l \big(x^{(k + \tfrac{l}{L})}\big)
\leq  
C_q d \big( x^{(k)}, x^{(k + \tfrac{l}{L})} \big) \left( 1+d\big( x^{(k)},q \big) \right).
\end{equation}
By the definition of the proximal mapping we have
\begin{equation*}
\varphi_l\big(x^{(k + \tfrac{l}{L})} \big) + \frac1{2\lambda_k} d\big(x^{(k + \tfrac{l-1}{L})},x^{(k + \tfrac{l}{L})} \big)^2  
\leq 
\varphi_l \big(x^{(k + \tfrac{l-1}{L})}\big)
\end{equation*}
and by~\eqref{eq:growthagain} further
\begin{align} 
d\big( x^{(k + \tfrac{l-1}{L})},x^{(k + \tfrac{l}{L})} \big) 
&\leq 
2\lambda_k \frac{\varphi_l \big( x^{(k + \tfrac{l-1}{L})} \big) - \varphi_l \big(x^{(k + \tfrac{l}{L})}\big)}{d\big(x^{(k + \tfrac{l-1}{L})} , x^{(k + \tfrac{l}{L})}\big)} \nonumber\\
& \leq 
2 \lambda_k C_q \Bigl( 1+d \big( x^{(k + \tfrac{l-1}{L})},q \big) \Bigr). \label{eq:mstep}
\end{align}
for every $l=1,\dots,L$.
For $l=1$ this becomes
\begin{equation} \label{eq:n=1}
 d\big(x^{(k)},x^{(k + \tfrac{1}{L})} \big) 
 \leq 
 2 \lambda_k C_q \left( 1+d\big( x^{(k)},q\big) \right),
\end{equation}
and for $l=2$ using~\eqref{eq:n=1} and the triangle inequality
\begin{align*} 
 d\big(x^{(k + \tfrac{1}{L})},x^{(k + \tfrac{2}{L})} \big) 
&\leq 
 2 \lambda_k C_q \Bigl( 1+d\big(x^{(k + \tfrac{1}{L})},q\big) \Bigr)\\
&\leq 
2 \lambda_k C_q \bigl(1+ 2 \lambda_k C_q \bigr) \Bigl( 1+d \big( x^{(k)},q \big) \Bigr).
\end{align*}
By~\eqref{eq:CPPAlambda} we can assume that $\lambda_k < 1$. Then replacing  $2 C_q \left(1+ 2 C_q \right)$
by a new constant which we call $C_q$ again, we get
\begin{equation*} 
 d\big(x^{(k + \tfrac{1}{L})},x^{(k + \tfrac{2}{L})} \big) 
\le 
 \lambda_k C_q \left( 1+d \big( x^{(k)},q \big) \right).
\end{equation*}
This argument can be applied recursively for $l=3,\dots,L$.
In the rest of the proof we will use $C_q$ as a generic constant independent of $\lambda_k$.
Using 
\begin{equation*} 
d \Big( x^{(k)}, x^{(k + \tfrac{l}{L})} \Big)
\leq 
d \Bigl( x^{(k)}, x^{(k + \tfrac{1}{L})} \Bigr)
+ \dots
+ d\bigl( x^{(k+ \tfrac{l-1}{L})}, x^{(k + \tfrac{l}{L})} \bigr) 
\end{equation*}
we obtain 
\begin{equation*} 
 d\big(x^{(k)},x^{(k + \tfrac{l}{L})} \big) \leq  \lambda_k C_q \left( 1+d\big( x^{(k)},q \big) \right),
\end{equation*}
for  $l=1,\dots,L$. Consequently we get by~\eqref{eq:growthagain_1} that
\begin{equation*} 
 \varphi_l \big(x^{(k)} \big)-\varphi_l \big(x^{(k + \tfrac{l}{L})} \big) 
 \leq  
  \lambda_k C_q \left( 1+d\big(x^{(k)},q\big)^2 \right),
\end{equation*}
for  $l=1,\dots,L$.
Plugging this inequality into~\eqref{eq:supermartingale1} yields
\begin{equation*}
d\big(x^{(k+1)},q \big)^2
\leq 
d \Big(x^{(k)},q \Bigr)^2 - 2\lambda_k \Bigl( \varphi \big(x^{(k)}\big)- \varphi(q) \Bigr) 
+ C_q \left( 1 +d \big(x^{(k)},q \big)^2 \right) 
\end{equation*}
which finishes the proof of~\eqref{eq:estimcyclic}.

\noindent 2. 
Assume now that $q \in {\mathcal H}$ is a minimizer of $\varphi$ 
and apply Lemma~\ref{lem:detmartingale} with 
$a_k \coloneqq d\big(x^{(k)},q \big)^2$, 
$b_k \coloneqq 2 \lambda_k \left( \varphi \big(x^{(k)}\big)- \varphi(q) \right)$,
$c_k \coloneqq C_q \lambda_k^2$ and $\eta_k \coloneqq C_q \lambda_k^2$
to conclude that the sequence
$
\{ d \big (x^{(k)},q \big)\}_{k\in\mathbb N_0}
$
converges and
\begin{equation} \label{eq:summable}
\sum_{k=0}^\infty \lambda_k\left( \varphi \big(x^{(k)}\big)- \varphi(q) \right) <\infty.
\end{equation}
In particular, the sequence $\{ x^{(k)}\}_{k \in \mathbb N}$ is bounded.
From~\eqref{eq:summable} and~\eqref{eq:CPPAlambda} we immediately obtain $\min \varphi = \liminf_{k \rightarrow \infty} \varphi \big(x^{(k)}\big)$, 
and thus there exists a cluster point $z\in {\mathcal H}$ of $\{ x^{(k)}\}_{k \in \mathbb N}$ which is a minimizer of $\varphi$.
Now convergence of $\{ d \big(x^{(k)},z \big)\}_{k\in\mathbb N_0}$ implies that$\{ x^{(k)}\}_{k \in \mathbb N}$  
converges to $z$ as $k\to\infty$.
(By~\eqref{eq:mstep} we see moreover that $\{ x^{(k+ \tfrac{l}{L})}\}_{k \in \mathbb N}$, $l=1,\dots,L$ converges to the same point.)
\end{proof}

Next we consider the inexact cyclic PPA which iteratively generates the points
$x^{(k+ \tfrac{l}{L})}$, $l=1,\ldots,L$, $k\in \mathbb N_0$,
fulfilling
\begin{equation} \label{inexact_PPA}
d\big(x^{(k+ \tfrac{l}{L})},\prox_{\lambda_k {\varphi}_l} (x^{( k + \tfrac{l-1}{L} )}) \big) < \frac{\varepsilon_k}{L},
\end{equation}
where $\{ \varepsilon_k \}_{k\in \mathbb N_0}$ is a given sequence of positive reals.
%-------------------------------------------------------------------------------
\begin{theorem}[Inexact Cyclic PPA] \label{thm:inexcyc}
Let $({\mathcal H},d)$ be a locally compact Hadamard space and let $\varphi$ be given by~\eqref{eq:f}.
Assume that for every starting point, 
the sequence generated by the exact cyclic PPA converges to a minimizer of~$\varphi$.
Let $\{x^{(k)} \}_{k \in \mathbb N}$ be the sequence generated by the inexact cyclic PPA in~\eqref{inexact_PPA}, 
where
$\sum_{k=0}^\infty \varepsilon_k <\infty$.  
Then the sequence $\{x^{(k)} \}_{k \in \mathbb N}$ converges to a minimizer of $\varphi$.
\end{theorem}
We note that the assumptions for Theorem~\ref{thm:inexcyc} are fulfilled if the assumptions of Theorem 4.3 are given.
%-------------------------------------------------------------------------------
%
\begin{proof}
For $k,m \in \mathbb N_0$, set 
\begin{equation*}
 y_{m,k} \coloneqq
 \left\{
 \begin{array}{ll}
  x^{(k)}
  &\mathrm{ if } \; k \le m,\\
  T_{k-1} (x^{(k-1)}) 
  &\mathrm{ if } \; k > m,
 \end{array}
 \right.
\end{equation*}
where
\[
T_{k} \coloneqq \prox_{\lambda_{k} {\varphi}_{L}} \circ \ldots \circ \prox_{\lambda_{k} {\varphi}_{1}}.
\]
Hence, for a fixed $m \in \mathbb N_0$, 
the sequence $\{y_{m,k}\}_k$ is obtained by inexact computations until the $m$-th step 
and by exact computations from the step $(m+1)$ on. 
In particular, the sequence 
$\{ y_{0,k}\}_k$ is the exact cyclic PPA sequence 
and 
$\{ y_{k,k}\}_k$ the inexact cyclic PPA sequence. 
By assumption we know that, 
for a given $m \in \mathbb N_0$, 
the sequence $\{y_{m,k}\}_k$ converges to minimizer $y_m$ of $\varphi$. 
Next we observe that the fact 
$\sum_{k=0}^\infty \varepsilon_k <\infty$ 
implies that the set 
$\{y_{m,k}\,:\, k,m \in \mathbb N_0\}$ is bounded:
Indeed, by our assumptions the sequence 
$\{ y_{0,k}\}_k$ 
onverges and therefore lies in a bounded set $D_0 \subset {\mathcal H}$. 
By~\eqref{inexact_PPA} and since the proximal mapping is nonexpansive, see \cite[Theorem 2.2.22]{mybook}, 
we obtain 
\begin{align}
 d\big( x^{(\frac{1}{L})}, \prox_{\lambda_0 \varphi_1} (x^{(0)}) \big) 
 &\le 
 \frac{\varepsilon_0}{L},\\
 d\left( x^{(\frac{2}{L})},  \prox_{\lambda_0 \varphi_2} \big( \prox_{\lambda_0 \varphi_1} (x^{(0)})  \big) \right)
 &\le
 d \big( x^{(\frac{2}{L})}, \prox_{\lambda_0 \varphi_2} (x^{(\frac{1}{L})}) \big)\\
 &+
 d\left( \prox_{\lambda_0 \varphi_2} (x^{(\frac{1}{L})}), \prox_{\lambda_0 \varphi_2} \big(\prox_{\lambda_0 \varphi_1} (x^{(0)})\big) \right)\\
 &\le
 \frac{\varepsilon_0}{L} + d\big(x^{(\frac{1}{L})}, \prox_{\lambda_0 \varphi_1} (x^{(0)}) \big)
 \le 
 \frac{2\varepsilon_0}{L}
\end{align}
and using the argument recursively
\[
d \big( x^{(1)} ,  T_0 (x^{(0)}) \big) \le \varepsilon_0.
\]
Hence $\{ y_{1,k}\}_k$ lies in a bounded set 
$D_1 \coloneqq \{x \in {\mathcal H} \colon d(x,D_0)<\varepsilon_0 \}$. 
The same argument yields that the sequence $\{ y_{m,k} \}_k$ with $m \ge 1$ lies in a bounded set 
\begin{equation*}
D_m \coloneqq \biggl\{ x \in {\mathcal H} \colon d(x,D_0) < \sum_{j=0}^{m-1} \varepsilon_j \biggr\}.
\end{equation*}
Finally the set $\{ y_{m,k} \colon k,m \in \mathbb N_0\}$ is contained in 
$\{ x \in {\mathcal H} \colon  d(x,D_0) < \sum_{j=0}^{\infty} \varepsilon_j \}$.

Consequently, also the sequence $\{ y_m\}_m$ is bounded 
and has at least one cluster point $z$ which is also a minimizer of $\varphi$.
Since the proximal mappings are nonexpansive, 
we have $d(y_m,y_{m+1})<\varepsilon_m$. 
Using again the fact that $\sum_{j=0}^\infty \varepsilon_j < \infty$, 
we obtain that the sequence 
$\{y_m\}_m$ cannot have two different cluster points 
and therefore $\lim_{m \rightarrow \infty} y_m = z$.

Next we will show that the sequence 
$\{x^{(k)}\}_k$ converges to $z$ as $k\rightarrow \infty.$ 
To this end, choose $\delta>0$ and find $m_1 \in \mathbb N$ 
such that 
$\sum_{j=m_1}^\infty \varepsilon_j<\frac\delta3$ 
and 
$d(y_{m_1},z)<\frac\delta3$. 
Next find $m_2>m_1$ such that whenever $m>m_2$ we have
\begin{equation*}
 d(y_{m_1,m},y_{m_1} )<\frac\delta3.
\end{equation*}
Since the proximal mappings is nonexpansive, we get
\begin{equation*}
 d(y_{m_1,m},y_{m,m} ) < \sum_{j=m_1}^{m-1} \varepsilon_j < \sum_{j=m_1}^\infty \varepsilon_j < \frac\delta3.
\end{equation*}
Finally, the triangle inequality gives
\begin{equation*}
 d(z,y_{m,m} ) < d(z,y_{m_1}) + d (y_{m_1},y_{m_1,m}) + d(y_{m_1,m},y_{m,m}) < \frac\delta3 + \frac\delta3 + \frac\delta3 =\delta
\end{equation*}
and the proof is complete. (For each $l=1,\dots,L$ the sequence $\{x^{(k+ \frac{l}{L})}\}_k$ has the same limit as $\{x^{(k)} \}_k$.)
\end{proof}
%-------------------------------------------------------------------------------
\begin{remark}
 Note that the condition $\sum_{k=0}^\infty \varepsilon_k<\infty$ is necessary. 
 Indeed, let $C \coloneqq \{(x,0)\in \mathbb R^2 \colon x \in \mathbb R\}$ 
 and let 
 $\varphi \coloneqq d(\cdot,C).$ 
 Then one can easily see that an inexact PPA sequence with errors $\varepsilon_k$ satisfying 
 $\sum_{k=0}^\infty \varepsilon_k  = \infty$ does not converge.
\end{remark}
\begin{remark}\label{rem:nonConv}
  While the theory of convergence for the inexact proximal point algorithm, especially the convergence Theorem 4.4 is valid, 
   we noticed that some functions in our splitting 
  from Section 4.1 do not fulfill the assumptions of the theorem.
  For the convergence of Algorithm 2 claimed in Corollary 4.6 of the former arXiv version, 
  all involved functions \(\varphi_l\) have to be geodesically convex. 
  Unfortunately the second order differences \(\mathrm{d}_2\) are not jointly convex in their three arguments as the following
  discussion shows. Let ${\cal M}$ be a finite dimensional Hadamard manifold.
  \begin{enumerate}[label=\roman*)]
   \item  Let \(x(t) \coloneqq \gamma_{\overset{\frown}{x_1,x_2}}(t)\) and
    \(z(t) \coloneqq \gamma_{\overset{\frown}{z_1,z_2}}(t)\) be two geodesics
    connecting \(x_1,x_2\) and \(z_1,z_2\) respectively, and 
    \(c(t) \coloneqq \gamma_{x(t),z(t)}(\frac{1}{2})\) the midpoint function.
    In particular we have
    $c(0) = \gamma_{\overset{\frown}{x_1,z_1}}(\frac12)$ and 
    $c(1) = \gamma_{\overset{\frown}{x_2,z_2}}(\frac12)$.
    In general this midpoint function does not coincide with the geodesic 
    $\gamma_c := \gamma_{\overset{\frown}{c(0),c(1)}}$.
    Take for example the Poincaré disc \(\mathcal M := \mathbb D\)  and
    \begin{align*}
    x_1 &\coloneqq \begin{pmatrix}
    \sin\frac{\pi}{3}\,\tanh 1\\\cos\frac{\pi}{3}\,\tanh 1
  \end{pmatrix},
  \quad
  x_2 \coloneqq \begin{pmatrix}
    -\sin\frac{\pi}{3}\,\tanh 1\\\cos\frac{\pi}{3}\,\tanh 1
  \end{pmatrix},\\
  z_1 &\coloneqq \begin{pmatrix}
    \sin\frac{\pi}{4}\,\tanh \frac{1}{2}\\[.5\baselineskip]
    \cos\frac{\pi}{4}\,\tanh \frac{1}{2}
  \end{pmatrix},
  \quad
  z_2 \coloneqq \begin{pmatrix}
    -\sin\frac{\pi}{4}\,\tanh\frac{1}{2}\\[.5\baselineskip]
    \cos\frac{\pi}{4}\,\tanh\frac{1}{2}
  \end{pmatrix}.
  \end{align*}
   The computed the mid point curve $c$ and the geodesic $\gamma_c$ are depiced  Figure~\ref{fig:geoCounter}.
   \item 
     The second order difference \(f(x,y,z) := \mathrm{d}_2(x,y,z)\) is in general not
    (jointly) convex.
    To this end, we use the example in i) and consider the second order difference function along
    $x(t)$,$ y(t) \coloneqq \gamma_c(t)$ and $z(t)$, $t \in [0,1]$.
    Since the mid point curve is not the geodesic, there exists a point \(t_0\in[0,1]\) such that
      \[
        \mathrm{d}_{\mathcal M}(c(t_0),\gamma_c(t_0)) > 0.
      \]
      Then we get 
      \[
          f(x(t_0),y(t_0),z(t_0)) = \mathrm{d}_{\mathcal M}(c(t_0),y(t_0)) = \mathrm{d}_{\mathcal M}( c(t_0),\gamma_c(t_0) ) >0 
      \]
      and 
      \begin{align*}
        (1-t_0) f(x(0),y(0),z(0)) + t_0 f(x(1),y(1),z(1))
        &= (1-t_0) \mathrm{d}_{\mathcal M}(c(x(0),z(0)),\gamma_c(0))\\
        &+ t_0 \mathrm{d}_{\mathcal M}(c(x(1),z(1)),\gamma_c(1)) = 0
      \end{align*}
      so that $f$ is not convex.
  \end{enumerate}
\end{remark}
  \begin{figure}\centering
        \includegraphics{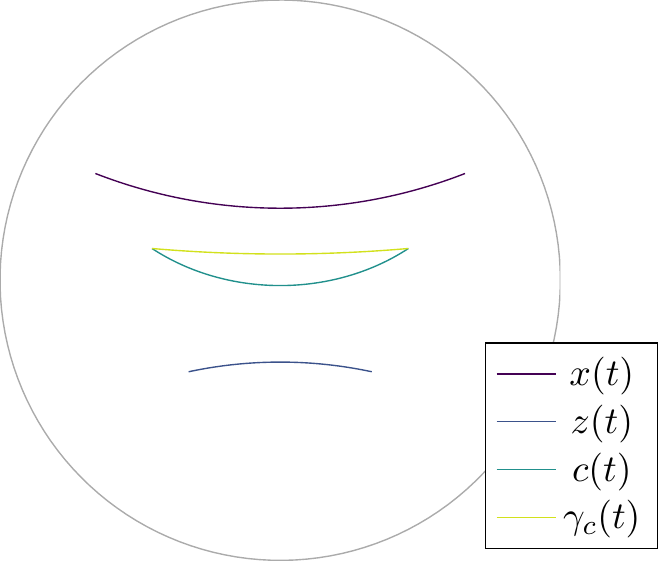}
    \caption{The geodesics \(x(t),z(t)\)  (violet) on \(\mathbb D\), its mid point curve $c$ (cyan) and the geodesic~\(\gamma_c\) 
     (light green) from the above example i) from Remark~\ref{rem:nonConv}. The curves $c$ and $\gamma_c$ do not coincide.}
    \label{fig:geoCounter}
  \end{figure}
\begin{remark}[Random PPA]\label{rem:rand-ppa}
Instead of considering the cyclic PPA in Theorem~\ref{thm:inexcyc},
one can study an inexact version of the random PPA, 
generalizing hence \cite[Theorem 3.7]{Bac14}. 
This would rely on the supermartingale convergence theorem and yield the almost sure convergence of the inexact PPA sequence. 
We however choose to focus on the cyclic variant and develop its inexact version, because it is appropriate for our applications.
\end{remark}

% ------------------------------------------------------------------------------
\section{Numerical Examples} \label{sec:numerics}
% ------------------------------------------------------------------------------
Algorithm~\ref{alg:CPPA} was implemented in \textsc{Matlab} and C++ with the Eigen library\footnote{available at \href{http://eigen.tuxfamily.org}{http://eigen.tuxfamily.org}} 
for both the sphere \(\mathbb S^2\) and the manifold of symmetric positive definite matrices \(\mathcal P(3)\) 
employing the subgradient method from Algorithm~\ref{alg:Subgrad}. In the latter algorithm we choose \(0\) from the subdifferential whenever it is multi-valued. 
Furthermore a suitable choice for the sequences in Algorithms~\ref{alg:CPPA} 
and~\ref{alg:Subgrad} is \(\lambda \coloneqq \{\tfrac{\lambda_0}{k}\}_k\), \(\lambda_0 = \tfrac{\pi}{2}\), 
and \(\tau \coloneqq \{\tfrac{\tau_0}{j}\}_j\), \(\tau_0=\lambda_k\), respectively.
The parameters in our model~\eqref{task_2} were chosen as $\alpha\coloneqq \alpha_1 = \alpha_2$ and 
$\beta \coloneqq \beta_1 = \beta_2 = \beta_3$ with an example depending grid search for an optimal choice. 
The experiments were conducted on a MacBook Pro running Mac OS X 10.10.3, Core
i5, 2.6\,GHz with 8 GB RAM using \textsc{Matlab} 2015a, Eigen 3.2.4 and 
the clang-602.0.49 compiler. 
For all experiments we set the convergence
criterion to \(1\,000\) iterations for one-dimensional signals 
and to \(400\) iterations for images. This yields the same number of proximal
mapping applied to each point, because we have \(L=15\) for the two-dimensional case and \(L=6\) in one dimension. To measure quality, we look at the
mean error
\[
E(x,y) = \frac{1}{\lvert \mathcal G\rvert}\sum_{i\in\mathcal G} d_{\mathcal M}(x_i,y_i)
\]
for two signals or images of manifold-valued data \(\{x_i\}_{i\in\mathcal G}, \{y_i\}_{i\in\mathcal G}\) defined on an index set \(\mathcal G\).

%-------------------------------------------------------------------------------
\subsection{\texorpdfstring{$\mathbb S^2$}{S\^2}-valued Data} \label{subsec:numerics:s2}
%-------------------------------------------------------------------------------
\paragraph{Sphere-Valued Signal.}
	\begin{figure}\centering
		\begin{subfigure}{.45\textwidth}\centering
			\includegraphics[width=.66\textwidth]{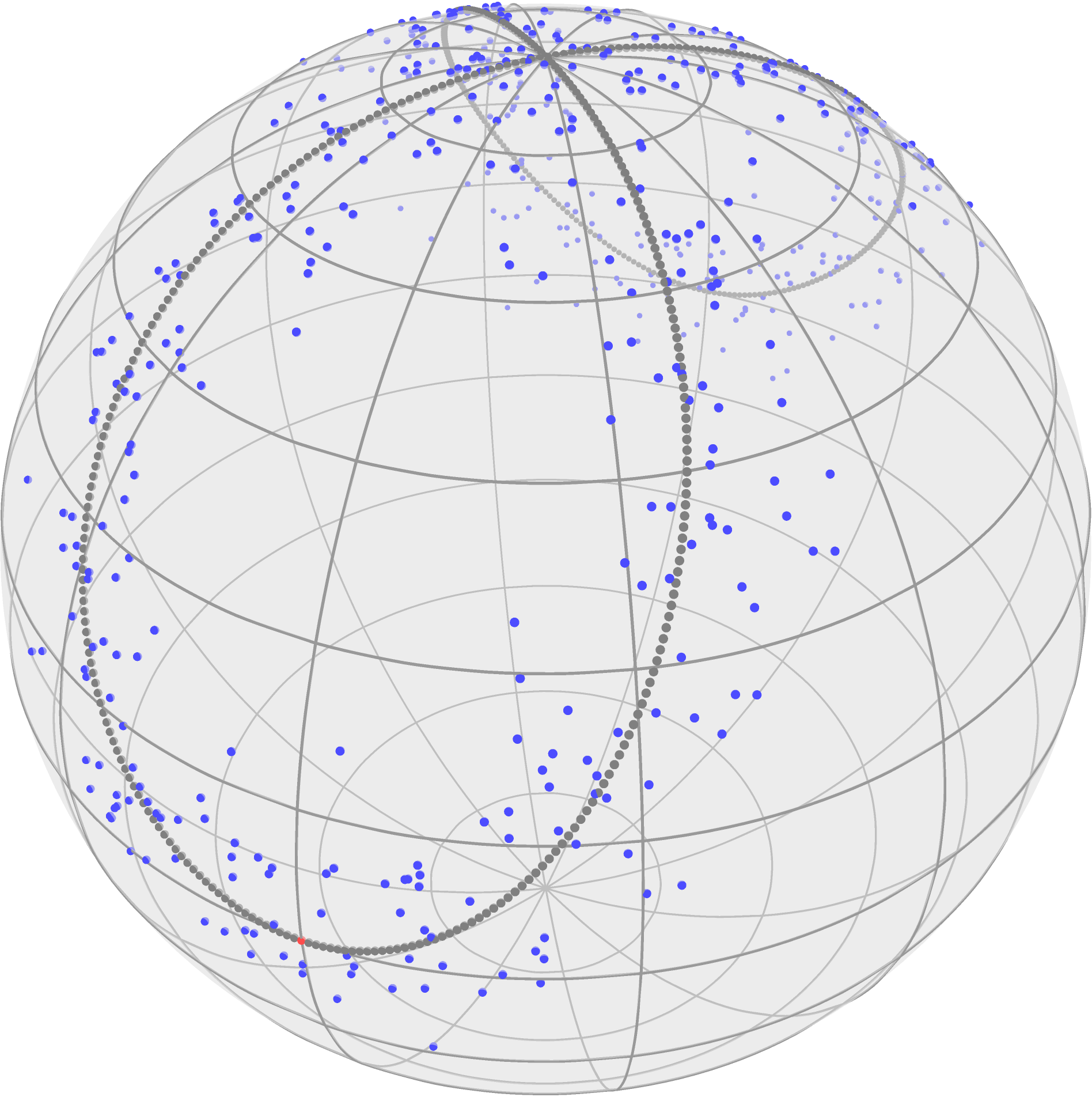}
			\caption{Noisy lemniscate of Bernoulli on \(\mathbb S^2\),\\Gaussian noise, \(\sigma=\frac{\pi}{30}\).}
			\label{subfig:lemniscate:noisy}
		\end{subfigure}
		\begin{subfigure}{.45\textwidth}\centering
			\includegraphics[width=.66\textwidth]{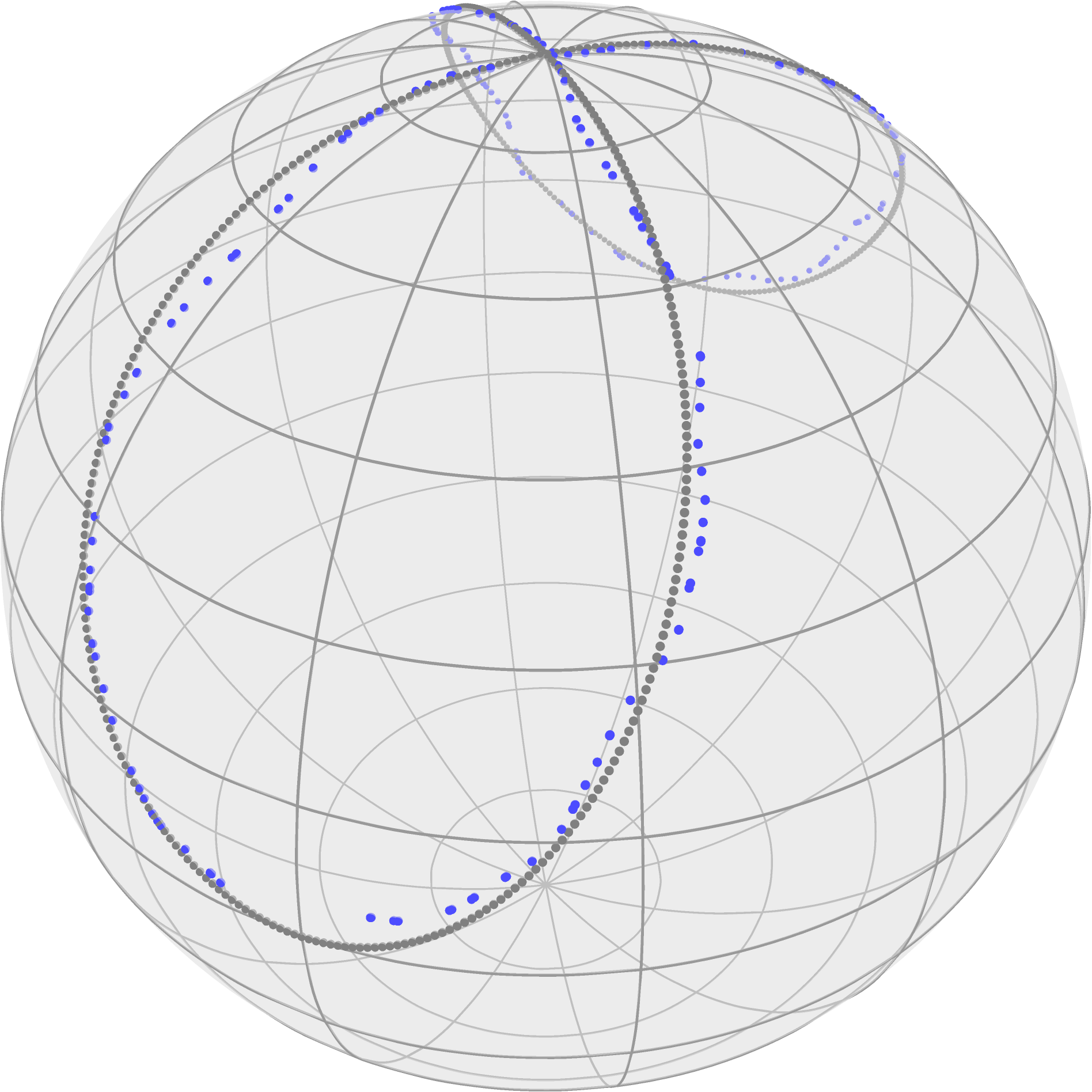}
			\caption{Reconstruction with \(\text{TV}_1\),\\\(\alpha=0.21\), \(E=4.08\times10^{-2}\).}
			\label{subfig:lemniscate:TV1}
		\end{subfigure}
		\begin{subfigure}{.45\textwidth}\centering
			\includegraphics[width=.66\textwidth]{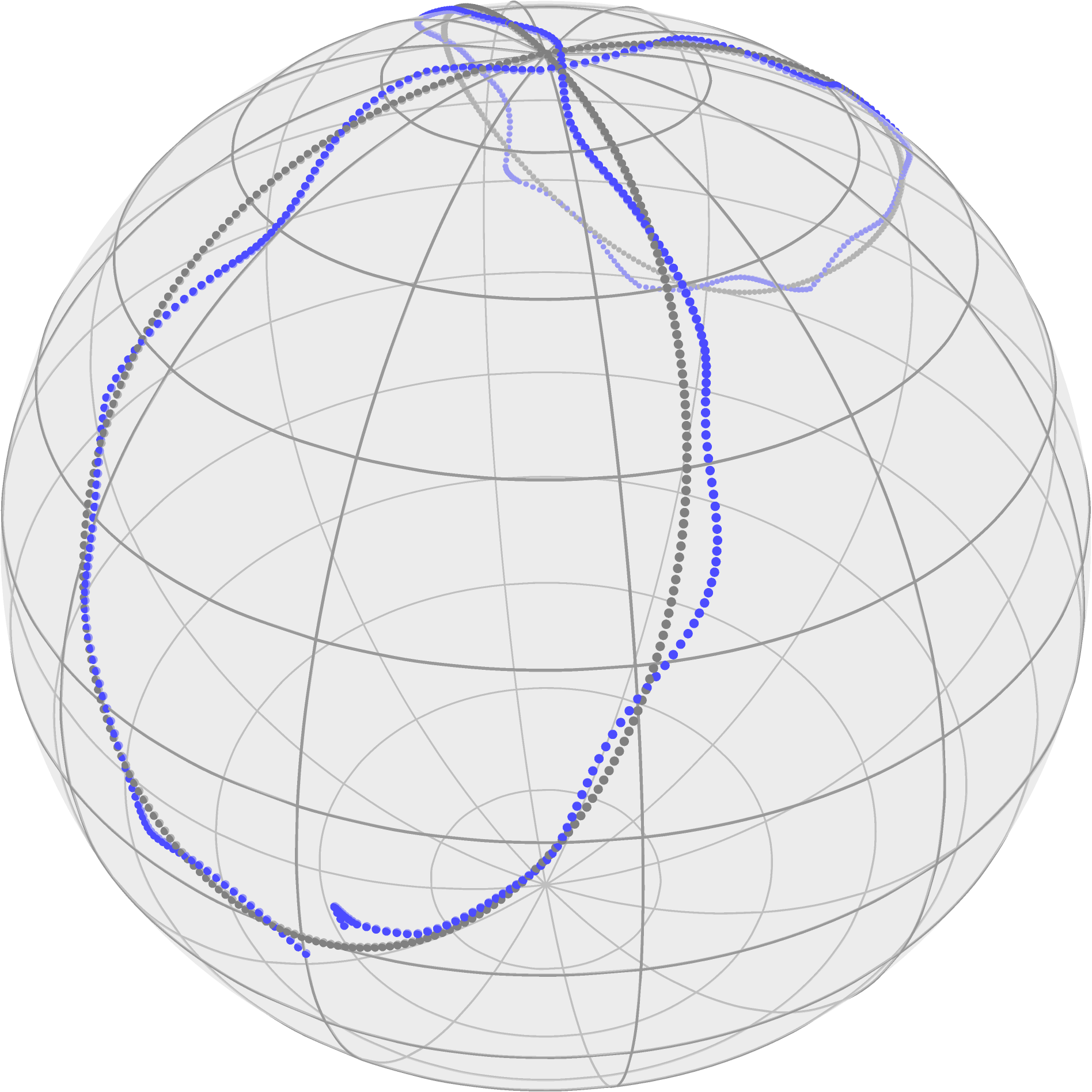}
			\caption{Reconstruction with \(\text{TV}_2\),\\
			\(\alpha=0\), \(\beta=10\), \(E=3.66\times10^{-2}\).}\label{subfig:lemniscate:TV2}
		\end{subfigure}
		\begin{subfigure}{.45\textwidth}\centering
			\includegraphics[width=.66\textwidth]{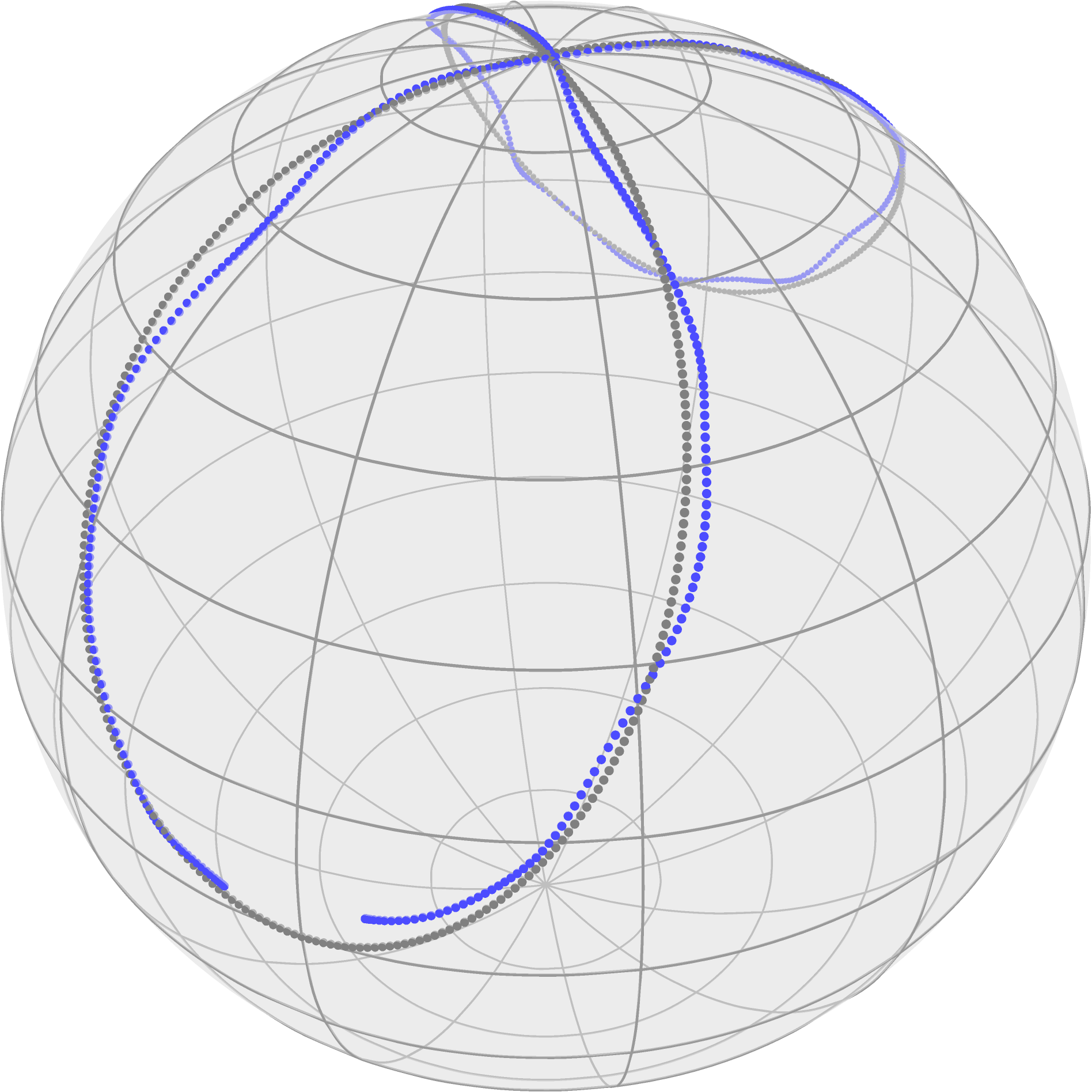}
			\caption{Reconstruction with \(\text{TV}_1\) \& \(\text{TV}_2\),\\
			\(\alpha=0.16\), \(\beta=12.4\), \(E=3.27\times10^{-2}\).}\label{subfig:lemniscate:TV12}
		\end{subfigure}
		\caption{Denoising an obstructed lemniscate of Bernoulli on the 
		sphere~\(\mathbb S ^2\).
		Combining first and second order differences yields the minimal value with respect to~\(E(f_{\text{o}},u_{\text{r}} )\).}
		\label{fig:lemniscate}
	\end{figure}
As first example we take a curve on the sphere~\(\mathbb S^2\). For any \(a>0\) the lemniscate of Bernoulli is defined as 
\[
	\gamma(t) \coloneqq
	\frac{a\sqrt{2}}{\sin^2(t)+1}\bigl(\cos(t),\cos(t)\sin(t)\bigr)^\tT, \quad 	t\in[0,2\pi]. 
\]
To obtain a curve on the sphere, we take an arbitrary point \(p\in\mathbb S^2\) and define the spherical lemniscate curve by
\[
	\gamma_S(t) = \log_p(\gamma(t))
\]
Setting \(a=\frac{\pi}{2\sqrt{2}}\), both extremal points of the lemniscate are antipodal, cf.\ the dotted gray line in Fig.~\ref{subfig:lemniscate:noisy}. 
We sample the spherical lemniscate curve for \(p=(0,0,1)^\tT\) at
\(t_i \coloneqq \frac{2\pi i}{511}\), \(i=0,\ldots,511\), to obtain a signal \(\bigl(f_{\text{o},i}\bigr)_{i=0}^{511}\in(\mathbb S^2)^{512}\). 
Note that the first and last point are identical. We colored them in red in Fig.~\ref{subfig:lemniscate:noisy}, 
where the curve starts counterclockwise, i.e., to the right.
This signal is affected by an additive 
Gaussian noise by setting~\(f_{i} \coloneqq \exp_{f_{\text{o},i}}\eta_i\) with \(\eta\) having standard
deviation of \(\sigma = \frac{\pi}{30}\) independently in both components. 
We obtain, e.g., the blue signal \(f=\bigl(f_{i}\bigr)_{i=0}^{511}\) in Fig.~\ref{subfig:lemniscate:noisy}.
We compare the TV regularization which was presented in~\cite{WDS2014} with our approach by measuring the mean error 
\( E(f_{\text{o}},u_{\text{r}})\) 
of the result~\(u_{\text{r}}\in(\mathbb S^2)^{511}\) 
to the original data $f_{\text{o}}$, which is always shown in gray.

The TV regularized result shown in Fig.~\ref{subfig:lemniscate:TV1} 
suffers from the well known staircasing effect, i.e., the signal is piecewise constant which yields groups of points having the same value and the signal to look sparser. 
The parameter was optimized with respect to \(E\) by a parameter search on \(\frac{1}{100}\mathbb N\) for \(\alpha\) and \(\frac{1}{10}\mathbb N\) for \(\beta\). 
When just using second order differences, i.e.\  setting~\(\alpha=0\), we obtain a better value for the quality measure, namely for \(\beta = 10\) we obtain \(E = 3.66\times10^{-2}\), 
see Fig.~\ref{subfig:lemniscate:TV2}. 
Combining the first and second order differences yields the best result with respect to\ \(E\), i.e.~\(E=3.27\times10^{-2}\) for \(\alpha=0.16\) and \(\beta=12.4\).

\paragraph{Two-Dimensional Sphere-Valued Data Example.}
\begin{figure}\centering
	\begin{subfigure}{.45\textwidth}\centering
		\includegraphics[width=.75\textwidth]{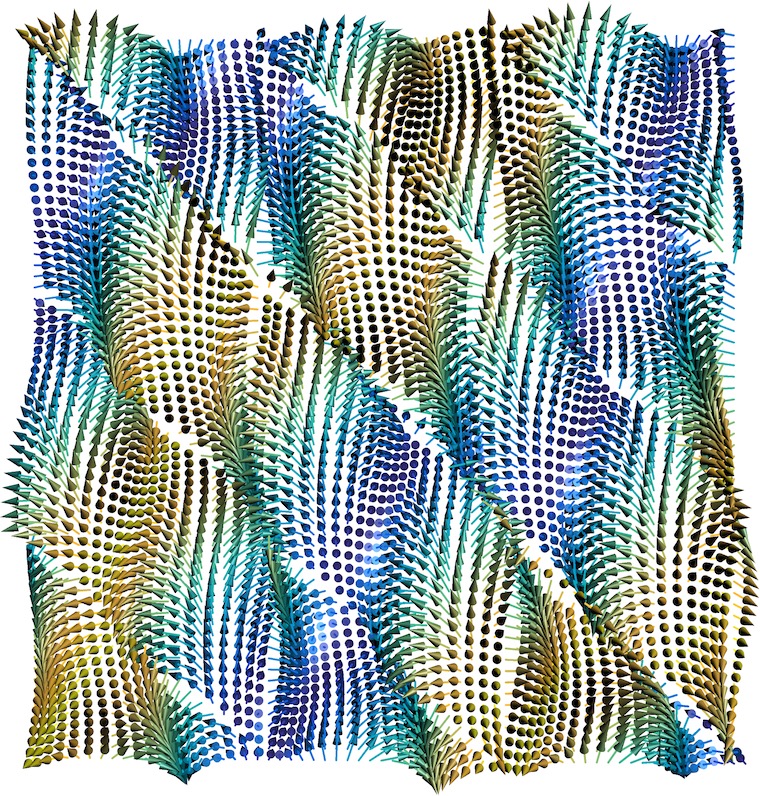}
		\caption{Original unit vector field.}
		\label{subfig:S2Field:Orig}
	\end{subfigure}
	\begin{subfigure}{.45\textwidth}\centering
		\includegraphics[width=.75\textwidth]{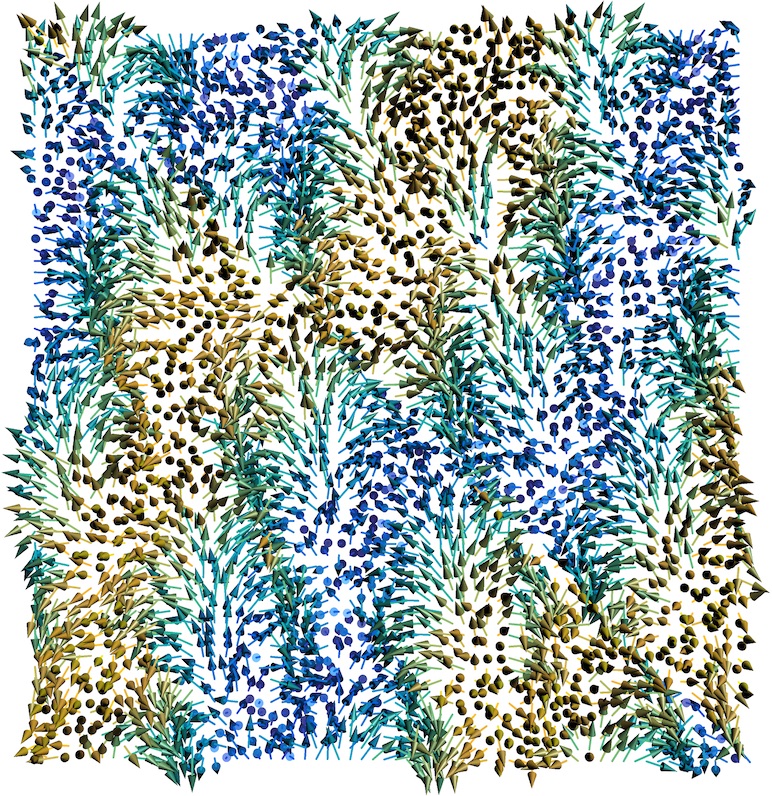}
		\caption{Noisy unit vector field.}
		\label{subfig:S2Field:Noisy}
	\end{subfigure}
	\begin{subfigure}{.45\textwidth}\centering
		\includegraphics[width=.75\textwidth]{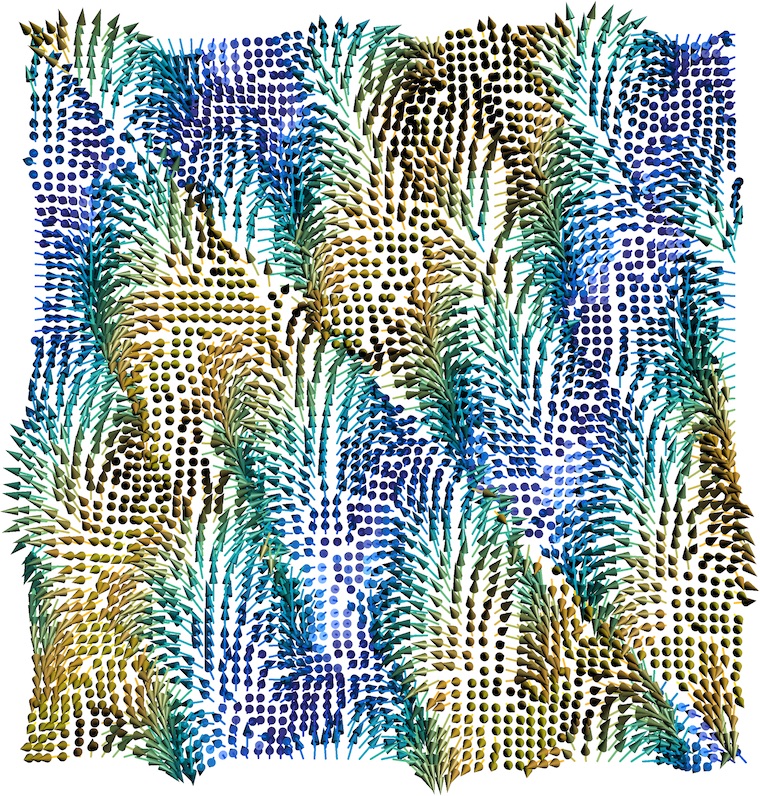}
		\caption{Reconstruction with \(\text{TV}_1\).}
		\label{subfig:S2Field:TV}
	\end{subfigure}
	\begin{subfigure}{.45\textwidth}\centering
		\includegraphics[width=.75\textwidth]{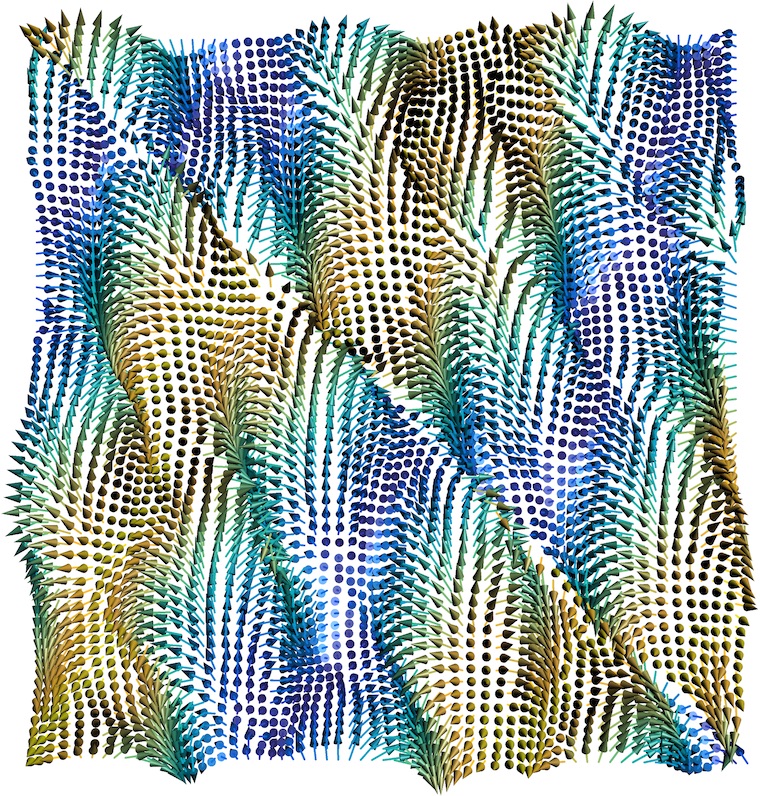}
		\caption{Reconstruction with \(\text{TV}_2\).}
	\label{subfig:S2Field:TV2}
	\end{subfigure}
	\begin{subfigure}{.975\textwidth}\centering
		\raisebox{-0.5\height}{\includegraphics[width=.1125\textwidth]{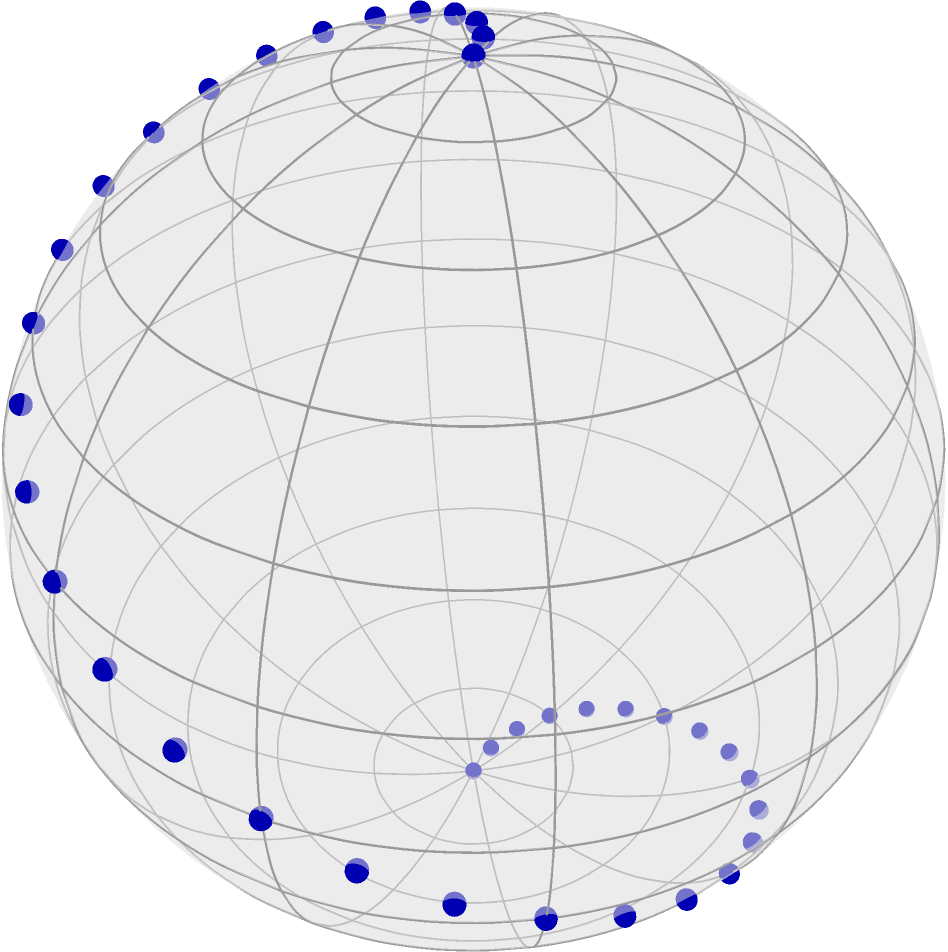}}
		\hspace{.02\textwidth}
		\raisebox{-0.5\height}{
		\includegraphics[width=.65\textwidth]{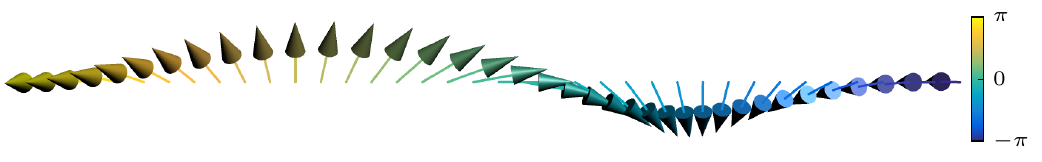}
		}
		\caption{Colormap illustration: a signal on \(\mathbb S^2\) (left) and drawn\\
		using arrows and the colormap \lstinline!parula! for elevation (right).}\label{subfig:S2field:colormap}
	\end{subfigure}
	\caption{Denoising results of an\subref{subfig:S2Field:Orig} \(\mathbb S^2\)-valued vector field, which is obstructed by\subref{subfig:S2Field:Noisy} Gaussian noise on \(T_{f_{\text{o}}}\mathbb S^2\), \(\sigma = \frac{4}{45}\pi\). A reconstruction using~\subref{subfig:S2Field:TV} \(\text{TV}_1\) approach, \(\alpha=3.5\times10^{-2}\), yields \(E=0.1879\) while~\subref{subfig:S2Field:TV2} the reconstruction with \(\text{TV}_2\), \(\alpha=0\), \(\beta=8.6\), yields an error of just \(E=0.1394\).}\label{fig:S2Field}
\end{figure}
We define an \(\mathbb S^2\)-valued vector-field by
\begin{align*}
	G(t,s) &= R_{t+s}S_{t-s}e_3,\quad t\in[0,5\pi], s\in[0,2\pi],\\
	&\text{ where }
	R_{\theta} \coloneqq \begin{pmatrix}
		\cos\theta&-\sin\theta&0\\\sin\theta&\cos\theta&0\\0&0&1
	\end{pmatrix},
	S_{\theta} \coloneqq \begin{pmatrix}
		\cos\theta&0&-\sin\theta\\0&1&0\\\sin\theta&0&\cos\theta\\
	\end{pmatrix}.
\end{align*}
We sample both dimensions with \(n=64\) points, and obtain a discrete vector
field \(f_{\text{o}}\in\bigl(\mathbb S^2\bigr)^{64\times 64}\) which is illustrated in Fig.~\ref{subfig:S2Field:Orig} the following way: on an
equispaced grid the point on \(\mathbb S^2\) is drawn as an arrow, where the
color emphasizes the elevation using the colormap \lstinline!parula! from \textsc{Matlab}, 
cf.~Fig.~\ref{subfig:S2field:colormap}. Similar to the sphere-valued signal, 
this vector field is affected by Gaussian noise imposed on the tangential plane at each point having a standard deviation of \(\sigma=\frac{5}{45}\pi\). 
The resulting noisy data~\(f\) is shown in Fig.~\ref{subfig:S2Field:Noisy}.

We again perform a parameter grid search on \(\frac{1}{200}\mathbb N\) to find good reconstructions of the noisy data, first for the denoising with first order difference terms (TV). 
For \(\alpha=3.5\times 10 ^{-2}\) we obtain the vector field shown in Fig.~\ref{subfig:S2Field:TV} having \(E=0.1879\). 
Introducing the complete functional from~\eqref{task_2}, we obtain setting \(\beta=8.6\) and \(\alpha=0\) an error of just \(E=0.1394\), 
see Fig.~\ref{subfig:S2Field:TV2}. 
Indeed, just using a second order difference term yields the best result here. 
Still, both methods cannot reconstruct the jumps along the diagonal lines from the original signal, 
because they vanish in noise. Only the main diagonal jump can roughly been recognized in both cases.

\paragraph{Application to Image Denoising.}
\begin{figure}\centering
	\begin{subfigure}{.31\textwidth}\centering
		\includegraphics[width=.95\textwidth]{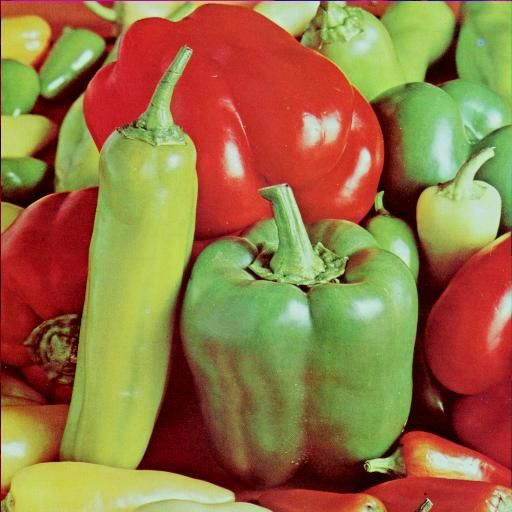}
		\caption{Original image.\\\ }\label{subfig:Peppers:orig}
	\end{subfigure}
	\begin{subfigure}{.31\textwidth}\centering
		\includegraphics[width=.95\textwidth]{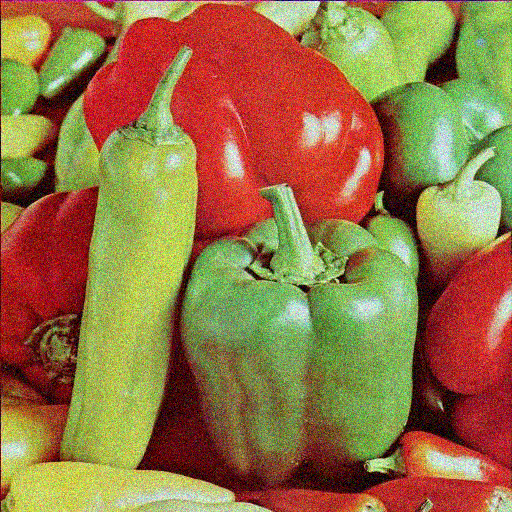}
		\caption{Noisy, \(\sigma=0.1\).\\\ }\label{subfig:Peppers:noisy}
	\end{subfigure}
	\begin{subfigure}{.31\textwidth}\centering
		\includegraphics[width=.95\textwidth]{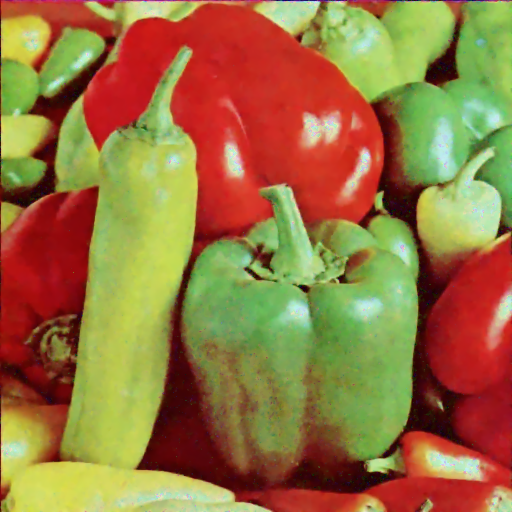}
		\caption{\(\text{TV}_1\)\&\(\text{TV}_2\),
		\\
		HSV, vectorial,\\}\label{subfig:Peppers:HSV}
	\end{subfigure}
	\begin{subfigure}{.31\textwidth}\centering
		\includegraphics[width=.95\textwidth]{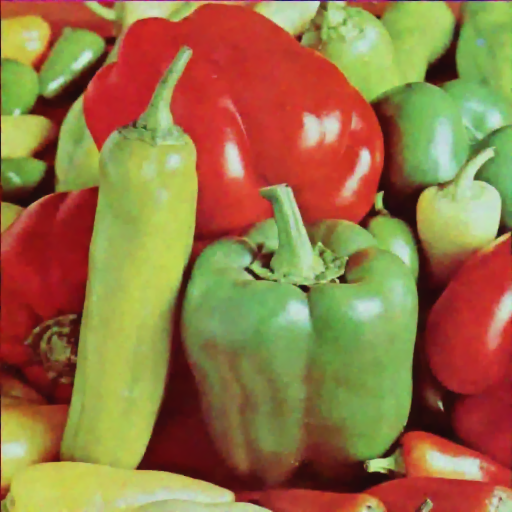}
		\caption{\(\text{TV}_1\)\&\(\text{TV}_2\),\\ RGB vectorial.}\label{subfig:Peppers:RGB}
	\end{subfigure}
	\begin{subfigure}{.31\textwidth}\centering
		\includegraphics[width=.95\textwidth]{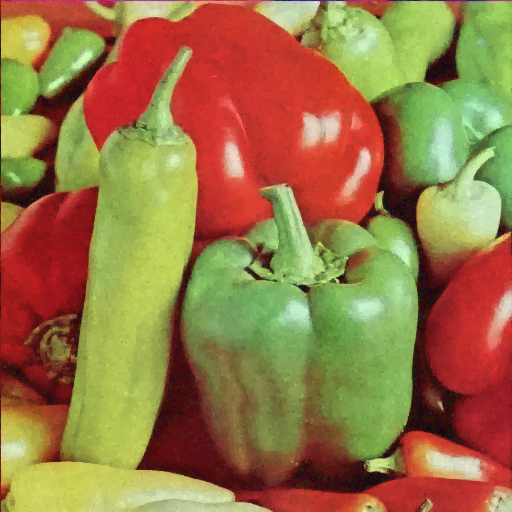}
		\caption{\(\text{TV}_1\),\\ CB, channel wise.}\label{subfig:Peppers:TV1onCB}
	\end{subfigure}
	\begin{subfigure}{.31\textwidth}\centering
		\includegraphics[width=.95\textwidth]{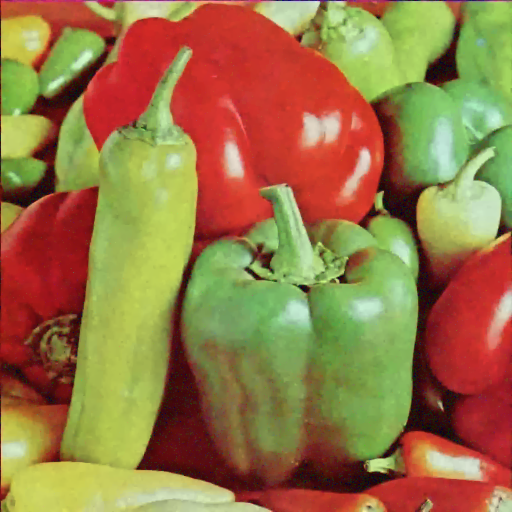}
		\caption{\(\text{TV}_1\)\&\(\text{TV}_2\),\\CB, channel wise.}\label{subfig:Peppers:TV12onCB}
	\end{subfigure}
	\caption{Denoising the “Peppers” image using approaches in various color spaces:\subref{subfig:Peppers:HSV} On HSV with an vectorial approach using \(\alpha=0.0625,\beta=0.125\) which yields a PSNR of \(28.155\),\subref{subfig:Peppers:RGB} on RGB with \(\alpha=0.05\), \(\beta=0.025\) yields a PSNR of \(31.241\), and two approaches channel wise on CB, where\subref{subfig:Peppers:TV1onCB} a TV approach with \(\alpha=0.05\) results in a PSNR of \(28.969\) and\subref{subfig:Peppers:TV12onCB} a \(\text{TV}_1\)\&\(\text{TV}_2\) approach, \(\alpha=0.024\), \(\beta=0.022\), yields a PSNR of \(29.7692\).}
\label{fig:Peppers}
\end{figure}
\enlargethispage{\baselineskip}
Next we deal with denoising in different color spaces.
Therefore we take the image “Peppers”\footnote{Taken from the USC-SIPI Image
Database, available online at \url{http://sipi.usc.edu/database/database. php?volume=misc&image=15}}, cf.~Fig.~\ref{subfig:Peppers:orig}. 
This image is distorted with Gaussian noise on each of the red, green and blue (RGB) channels
with \(\sigma = 0.1\), cf.~Fig.~\ref{subfig:Peppers:noisy}. 
Besides the RGB space, we consider the Hue-Value-Saturation (HSV) color space
consisting of a \(\mathbb S^1\)-valued hue component \(H\)~and two real valued
components~\(S,V\). For the latter one, there are many methods, e.g., vector
valued TV. For both, the authors presented a vector-valued first and second order TV-type approach in~\cite{BW15a,BW15b}.\enlargethispage{\baselineskip}
We compare the vectorial approaches to  the Chromaticity-Brightness (CB), 
where we apply a second order TV on the real-valued brightness and the \(\mathbb S^2\)-valued chromaticity separately.
To be precise, the obtained chromaticity values are in the positive octant of \(\mathbb S^2\). 
Again we search for the best value ---here with respect to PSNR--- of the denoising models at hand on a grid of \(\frac{1}{500}\mathbb N\) 
for the available parameters. For the component based approach of CB, both components are treated with the same parameters.

While for this example already the TV-based approach on the separate channels C and B outperforms the HSV vectorial approach, 
both the TV and the combined first and second order approach on CB are outperformed by the vectorial RGB approach. 
The reason for that is, that both channels of brightness and chromaticity in the latter model are not coupled.
It would be interesting to couple the channels in the CB color model in a future work.
%-----------------------------------------
\subsection{\texorpdfstring{${\mathcal P} (3)$}{P(3)}-valued Images} \label{subsec:Pr}
\enlargethispage{3\baselineskip}
%-----------------------------------------
\paragraph{An Artificial Matrix-Valued Image.}
\begin{figure}\centering
	\begin{subfigure}{.45\textwidth}\centering
		\includegraphics[width=.6\textwidth]{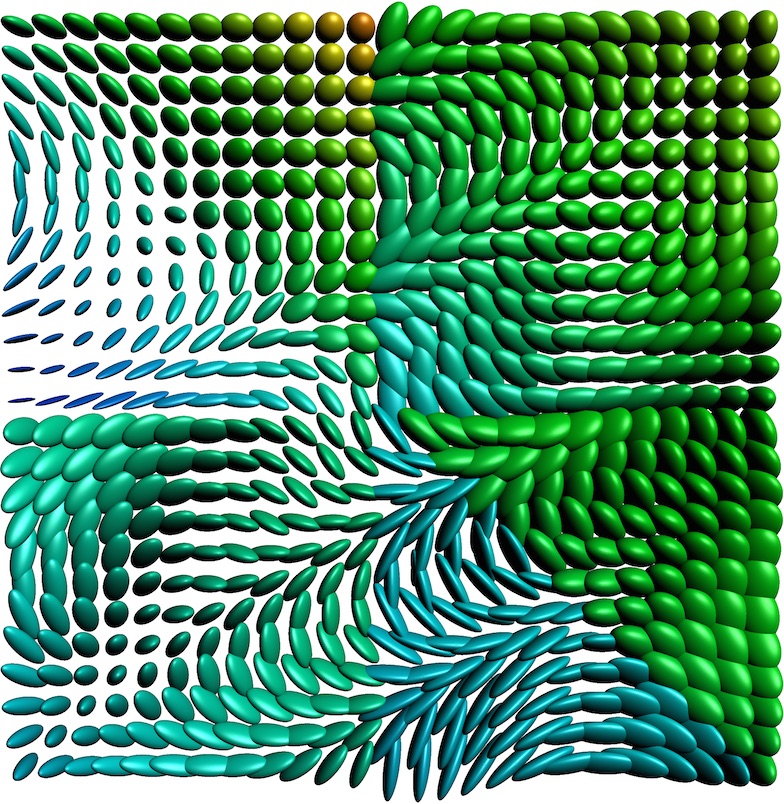}
		\caption{Original Data.\\\ }\label{subfig:ArtSPD:orig}
	\end{subfigure}
	\begin{subfigure}{.45\textwidth}\centering
		\includegraphics[width=.6\textwidth]{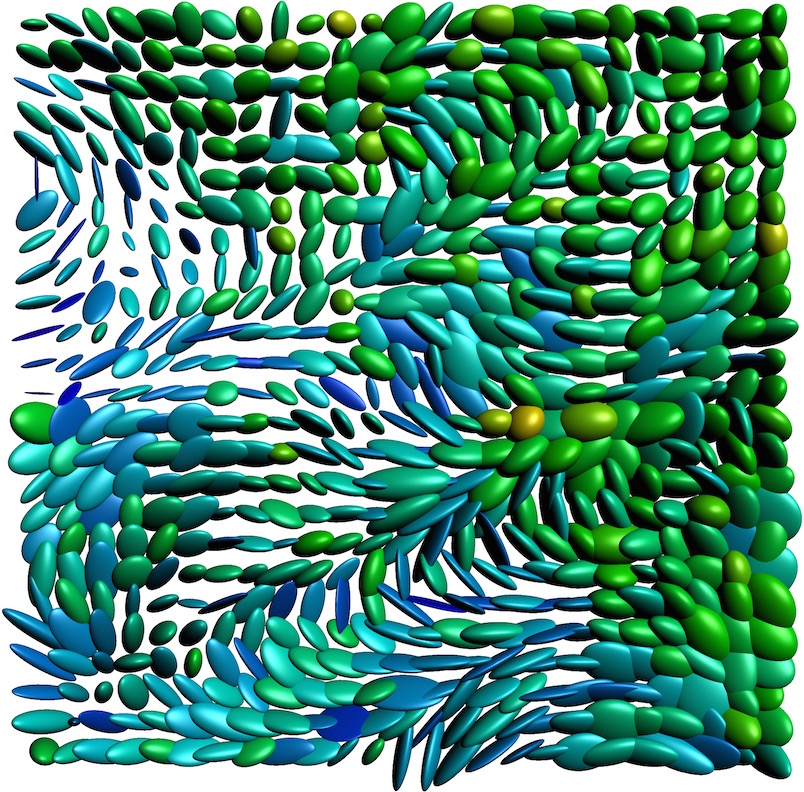}
		\caption{Noisy Data,\\Rician noise, \(\sigma=0.03\).}\label{subfig:ArtSPD:noisy}
	\end{subfigure}
	\begin{subfigure}{.45\textwidth}\centering
		\includegraphics[width=.6\textwidth]{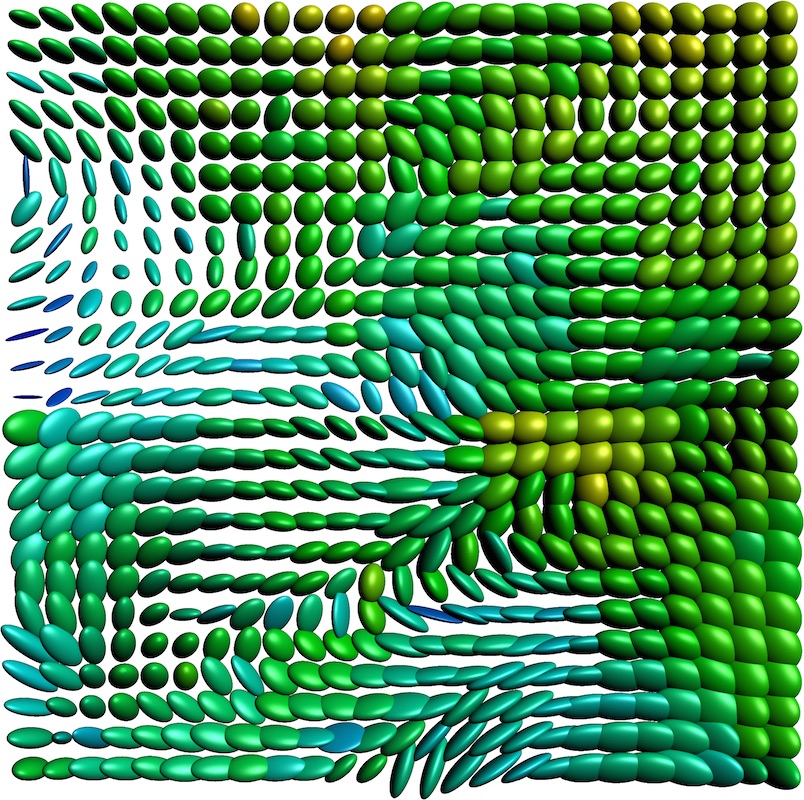}
		\caption{Reconstruction with \(\text{TV}_1\),\\
		\(\alpha=0.1\), \(E=0.4088\).}\label{subfig:ArtSPD:TV}
	\end{subfigure}
	\begin{subfigure}{.45\textwidth}\centering
		\includegraphics[width=.6\textwidth]{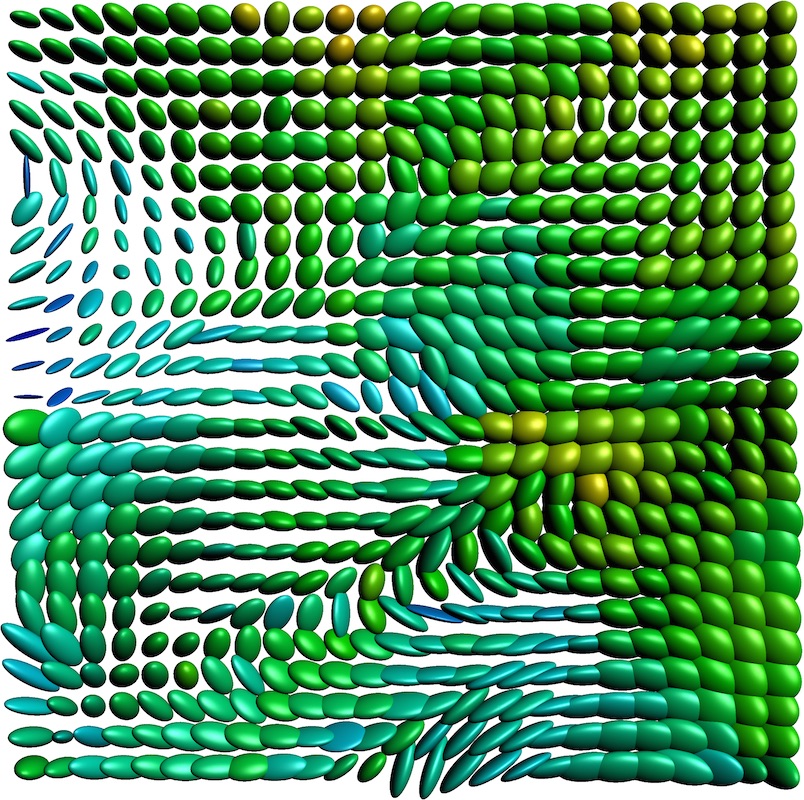}
		\caption{Reconstruction with \(\text{TV}_1\) \& \(\text{TV}_2\),\\\ 
		\(\alpha=0.035\), \(\alpha=0.02\), \(E=0.4065\).}\label{subfig:ArtSPD:TV2}
	\end{subfigure}
	\caption{Denoising an artificial image of SPD-valued data.}
	\label{fig:ArtSPD}
\end{figure}
We construct an artificial image of \(\mathcal P(3)\)-valued pixels by sampling
\begin{align*}
	G(s,t)
	&\coloneqq
		A(s,t)\operatorname{diag}
		\begin{pmatrix}
		1+\delta_{x+y,1}\\
		1+s+t+\frac{3}{2}\delta_{s,\frac{1}{2}}\\
		4-s-t+\frac{3}{2}\delta_{t,\frac{1}{2}}
		\end{pmatrix}
		A(s,t)^\tT,\qquad s,t\in[0,1],\\
	&\quad\text{where }
	A(s,t) =
	R_{x_2,x_3}(\pi s)
	R_{x_1,x_2}(\lvert2\pi s-\pi\rvert)
	R_{x_1,x_2}\bigl(\bigr\lvert \pi(t-s-\lfloor t-s\rfloor)-\pi\bigr\rvert\bigr),\\
	&\qquad R_{x_i,x_j}(t)\text{ rotation in the }x_i,x_j\text{-plane and }\delta_{a,b} = \begin{cases} 1&\text{ if } a>b\\0&\text{ else.}
	\end{cases}
\end{align*}
Despite the outer rotations the diagonal, i.e., the eigenvalues introduce three
jumps along both center vertical and horizontal lines and along the diagonal, 
see Fig.~\ref{subfig:ArtSPD:orig}, where this
function is sampled to obtain an \(25\times25\) matrix valued image \(f = \bigl(f_{i,j}\bigr)_{i,j=1}^{25} \in \mathcal P(3)^{25,25}\). 
We visualize any symmetric positive definite matrix \(f_{i,j}\) 
by drawing a shifted ellipsoid given by the surface niveau \(\{x\in\mathbb R^3\,:\,(x^\tT - c(i, j))f_{i,j}(x - c(i, j)^\tT) = 1\}\) 
for some grid scaling parameter \(c>0\). 
As coloring we use the  anisotropy index relative to the Riemannian distance~\cite{MoBa06} normalized onto \([0,1)\), 
which is also known as the geodesic anisotropy index. 
Together with the hue color map from \textsc{Matlab} both the unit matrix yielding a sphere 
and the case where one eigenvalue dominates by far get colored in red. 

\paragraph{Application to DT-MRI.}
\begin{figure}\centering
	\begin{subfigure}{.49\textwidth}\centering
		\includegraphics[width=.975\textwidth]{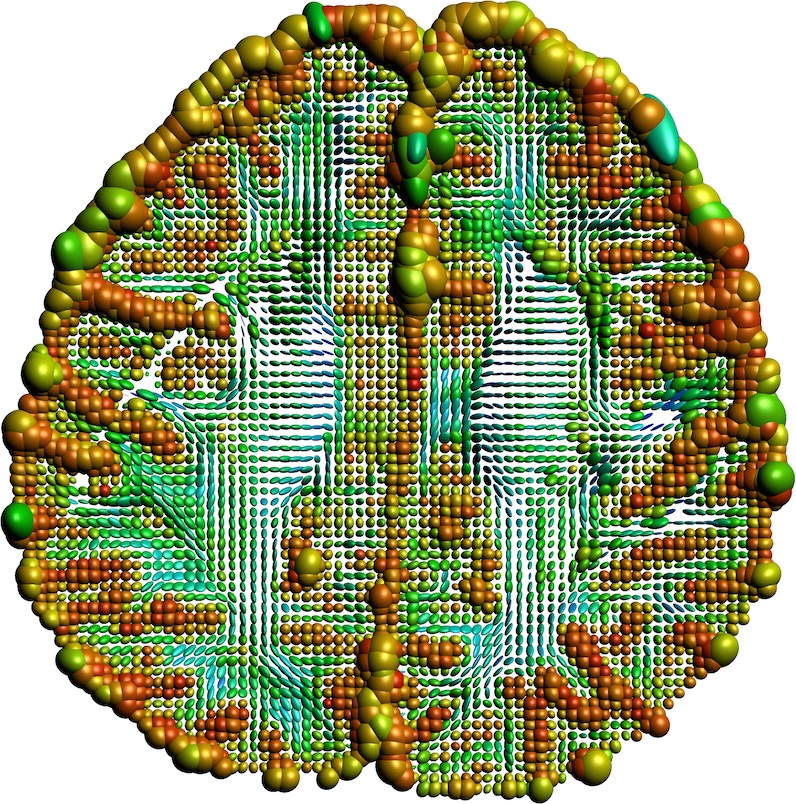}
		\caption{Original data.\\\ }%
		\label{subfig:CaminoC:orig}
	\end{subfigure}
	\begin{subfigure}{.49\textwidth}\centering
		\includegraphics[width=.975\textwidth]{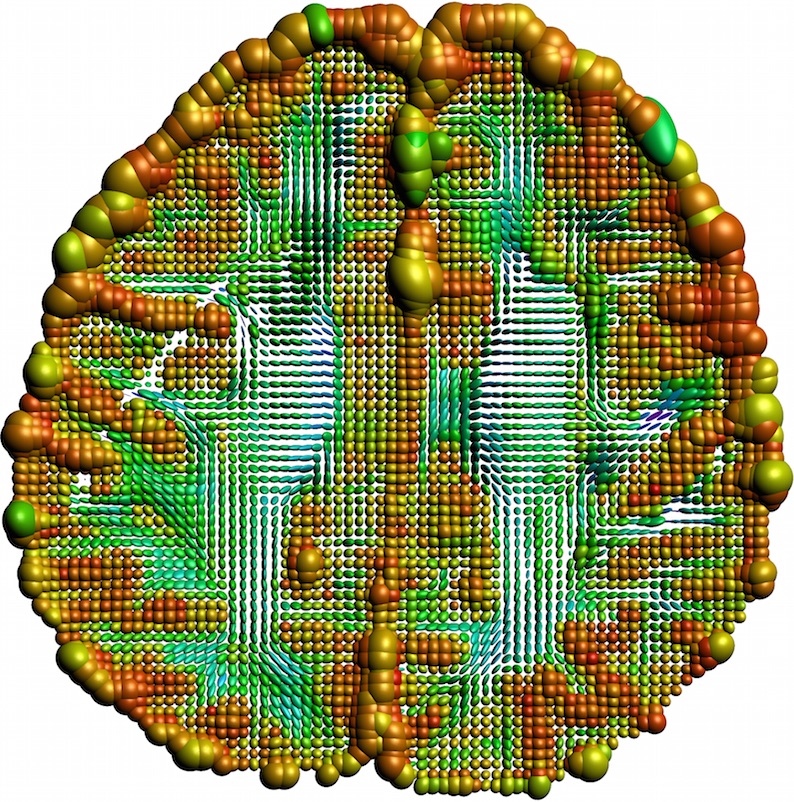}
		\caption{Reconstruction with \(\text{TV}_1\) \& \(\text{TV}_2\), \(\alpha=0.01\), \(\beta=0.05\).}%
		\label{subfig:CaminoC:reg}
	\end{subfigure}
	\begin{subfigure}{.49\textwidth}\centering
		\includegraphics[width=\textwidth]{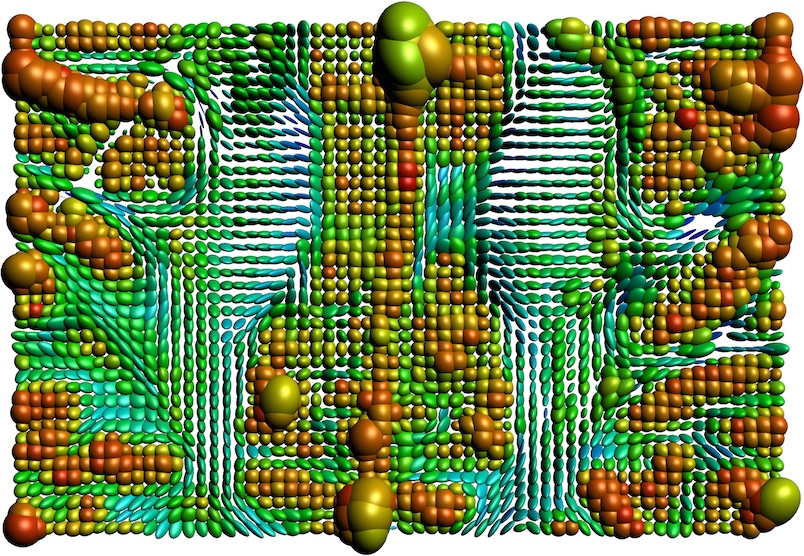}
		\caption{
		Subset of \subref{subfig:CaminoC:orig}.
		}\label{subfig:Camino:orig}
	\end{subfigure}
	\begin{subfigure}{.49\textwidth}\centering
		\includegraphics[width=\textwidth]{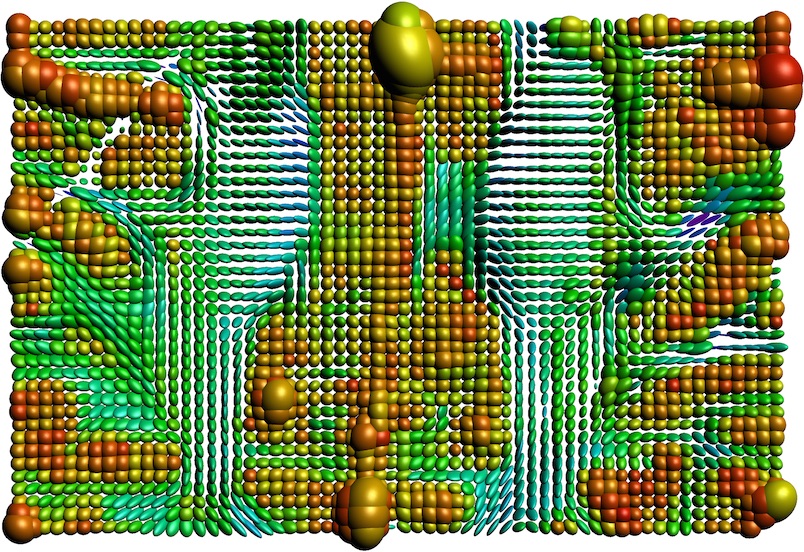}
		\caption{Subset of \subref{subfig:CaminoC:reg}.}%
		\label{subfig:Camino:reg}
	\end{subfigure}
	\caption{The Camino DT-MRI data of slice \(28\):\subref{subfig:CaminoC:orig} the original data,\subref{subfig:CaminoC:reg} the \(\text{TV}_1\)\&\(\text{TV}_2\)-regularized model keeps the main features but smoothes the data. For both data
	a subset
	is shown in\subref{subfig:Camino:orig} and\subref{subfig:Camino:reg}, respectively.}
	\label{fig:CaminoC}
\end{figure}
Finally we explore the capabilities of applying the denoising technique to real world data. 
The Camino project\footnote{see~\href{http://cmic.cs.ucl.ac.uk/camino/}{http://cmic.cs.ucl.ac.uk/camino}}\cite{Camino} 
provides a dataset of a Diffusion Tensor Magnetic Resonance Image (DT-MRI) of the human head, 
which is freely available.\footnote{follow the tutorial at \href{http://cmic.cs.ussscl.ac.uk/camino//index.php?n=Tutorials.DTI}{http://cmic.cs.ucl.ac.uk/camino//index.php?n=Tutorials.DTI}}
From the complete dataset of \(\tilde f = \bigl(\tilde f_{i,j,k}\bigr)\in\mathcal P(3)^{112\times112\times50}\) we take the traversal plane \(k=28\), see Fig.~\ref{subfig:CaminoC:orig}.
By combining a first and second order model for denoising, the noisy parts are reduced, while constant parts as well as basic features are kept, see Fig.~\ref{subfig:CaminoC:reg}. 
To see more detail, we focus on the subset \((i,j)\in\{28,...,87\}\times\{24,\ldots,73\}\),
which is shown in Figs.~\ref{subfig:Camino:orig} and\,\subref{subfig:Camino:reg}, respectively.
%-------------------------------------------------------------------------------
\paragraph{Acknowledgement.} This research was partly conducted when AW visited the University Kaiserslautern 
and when MB and RB were visiting the Helmholtz-Zentrum M\"unchen. We would like
to thank A. Trouv\'e for valuable discussions.
AW is supported by the Helmholtz Association within the young
investigator group VH-NG-526. AW also acknowledges the support
by the DFG scientific network ``Mathematical Methods in Magnetic Particle Imaging''. 
GS acknowledges the financial support by DFG Grant STE571/11-1.
\appendix
%-------------------------------------------------------------------------------
\section{The Sphere \texorpdfstring{${\mathbb S}^2$}{S2}} \label{A-sphere}
%-------------------------------------------------------------------------------
We use the parametrization
\[
x(\theta , \varphi) = 
\begin{pmatrix}
\cos \varphi \cos \theta,
\sin \varphi \cos \theta,
             \sin \theta
\end{pmatrix}^T, \quad \theta \in \Bigl(-\frac{\pi}{2}, \frac{\pi}{2}\Bigr), \; \varphi \in [0,2\pi)
\]
with north pole $z_0 \coloneqq (0,0,1)^\tT = x(\frac{\pi}{2}, \varphi_0)$ and south pole $(0,0,-1)^\tT = x(-\frac{\pi}{2},\varphi_0)$.
Then we have for the tangent spaces
\[T_x(\mathbb S^d) = T_{x(\theta,\varphi)}(\mathbb S^2)\coloneqq \{v \in \mathbb R^{d+1} : v^\tT x = 0\}
= \vspan\{e_1(\theta,\varphi), e_2(\theta,\varphi)\}\]
with the normed orthogonal vectors
\(
e_1(\theta,\varphi) = \frac{\partial x}{\partial \theta} = 
(\cos \varphi \sin \theta,
- \sin \varphi \sin \theta,
             \cos \theta)^\tT
\)
and
\(
e_2(\theta,\varphi) = \frac{1}{\cos \theta} \frac{\partial x}{\partial \varphi} = 
(- \sin \varphi, 
  \cos \varphi,
             0)^\tT
\).
The geodesic distance is given by
$d_{\mathbb S^2}(x_1,x_2) = \arccos (x_1^\tT x_2)$,
the unit speed geodesic by
\( 
\gamma_{x,\eta}(t)
= \cos(t) x + \sin(t) \frac{\eta}{\lVert\eta\rVert_2}, 
\) and the exponential map by
\begin{align}
\exp_{x}(t\eta) &= \cos(t\lVert\eta\rVert_2) x + \sin(t \lVert\eta \rVert_2) \frac{\eta}{\lVert\eta\rVert_2}.
\end{align}
Finally, a unit speed geodesic trough $x$ (with $\theta \ge 0$) and the north pole $z_0$ reads
\[
\gamma_{x,e_1}(t) = \cos(t) x + \sin(t) e_1,
\quad \gamma_{x,e_1}(T) = z_0 ,
\quad c(x) = \gamma_{x,e_1}(\tfrac{T}{2}),
\quad T = \frac{\pi}{2} - \theta_x
\]
and the orthogonal frame  along this geodesic as
$
E_1(t) = e_1 \left(\theta(t), \varphi \right)
$,
$
E_2(t) = e_2 \left(\theta(t) , \varphi \right)
$, 
$
\theta(t) \coloneqq \theta_x + \frac{t}{T} \left(\frac{\pi}{2} - \theta_x \right).
$
%
%-------------------------------------------------------------------------------
\section{The Manifold \texorpdfstring{$\mathcal P(r)$}{P(r)} of Symmetric Positive Definite Matrices}\label{A-spd}
%---------------------------------------------------------------------
We provide definitions 
for the manifold of  positive definite symmetric $r \times r$ matrices ${\mathcal P}(r)$
which are required in our computations, see \cite{SH14}.
By $\Exp$ and $\Log$ we denote the matrix exponential and logarithm defined by
\(
\Exp x \coloneqq \sum_{k=0}^\infty \frac{1}{k!} x^k
\)
and
\( \Log x \coloneqq \sum_{k=1}^\infty \frac{1}{k} (I-x)^k, \quad \rho(I-x) < 1.
\)
The geodesic distance is given by
\[
d_{\mathcal P} (x_1,x_2) \coloneqq \lVert\Log(x_1^{-\frac12} x_2 x_1^{-\frac12} )\rVert.
\]
Further, we have the exponential map
\(
\exp_x (t\eta) \coloneqq x^{\frac12} \Exp(t x^{-\frac12} \eta x^{-\frac12}) x^{\frac12}.
\)
The unit speed geodesic linking $x$ and $z$ for $t=T =d_{\mathcal P} (x,z)$ is
\[
\gamma_{\overset{\frown}{x,z} }(t) = \exp_x(t v) \coloneqq x^{\frac12} \Exp \left(\tfrac{t}{T} \Log (x^{-\frac12} z x^{-\frac12} ) \right) x^{\frac12},
\]
where
\(
v = x^{\frac12} \Log (x^{-\frac12} z x^{-\frac12}) x^{\frac12}/\lVert\Log (x^{-\frac12} z x^{-\frac12})\rVert.
\)
In particular we obtain
\[
c(x,z) = \gamma_{\overset{\frown}{x,z}} (\frac12) = x^{\frac12} (x^{-\frac12} z x^{-\frac12})^{\frac12}  x^{\frac12}
\]
which is known as the geometric mean of $x$ and $z$.
The parallel transport of $\eta \in T_x{\mathcal P}$ along the geodesic $\gamma_{x,\xi}(t) = \exp_x( t \xi)$ is given by
\( P_t(\eta) = \exp_x (\tfrac{t}{2} \xi) \, x^{-1} \eta x^{-1} \exp_x (\tfrac{t}{2} \xi) \).

%-------------------------------------------------------------------------------
\section{Basics on Parallel Transport} \label{C-transport}
%-------------------------------------------------------------------------------
\begin{figure}\centering
	\includegraphics{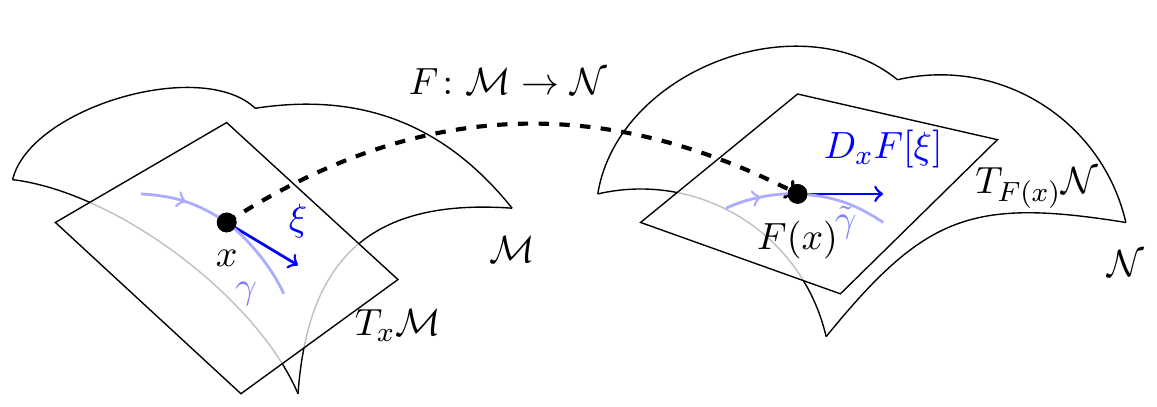}
	\caption{Illustration of the differential of \(F\colon\mathcal M\to\mathcal N\). Here \(\tilde \gamma = F\circ\gamma\).}
	\label{fig:differential}
\end{figure}
In this section we review some concepts from differential geometry which were used
in the paper. For more details we refer to \cite{Lee97}.
Let 
${\mathcal C}^\infty({\mathcal M})$ 
denote the set of smooth
real-valued functions on a manifold ${\mathcal M}$
and
${\mathcal C}^\infty(x)$ 
the functions defined on some open  neighborhood of $x \in  {\mathcal M}$ which are smooth at $x$.
Further, let  ${\mathcal C}^\infty({\mathcal M},{\mathcal N})$ denote the smooth maps from ${\mathcal M}$ to a manifold ${\mathcal N}$.
For computational purposes we introduce tangent vectors by their curve realizations.
A tangent vector $\xi$ to a manifold ${\mathcal M}$ at $x \in {\mathcal M}$ 
is a mapping from to ${\mathcal C}^\infty(x)$ to $\mathbb R$ such that there exists a 
curve $\gamma\colon\mathbb R \rightarrow {\mathcal M}$ with $\gamma(0) = x$
satisfying 
\[
	\xi f = \dot \gamma (0) f \coloneqq \frac{d}{dt} f(\gamma(t)) \bigr|_{t=0}.
\]
The set of all tangent vectors at $x \in {\mathcal M}$ forms the tangent space $T_x{\mathcal M}$.
Further,  $T{\mathcal M}\coloneqq\cup_x\,T_x{\mathcal M}$ is the tangent bundle of ${\mathcal M}$. 
Given  $F \in {\mathcal C}^\infty({\mathcal M},{\mathcal N})$ and a curve 
$\gamma\colon (-\varepsilon,\varepsilon) \rightarrow {\mathcal M}$ with $\gamma(0) = x$ and $\dot \gamma (0) = \xi$,
then 
\begin{align} \label{def:differential}
D_x F\colon T_x{\mathcal M} \rightarrow T_{F(x)} {\mathcal N}, \quad  \xi \mapsto  D_x F[\xi] \coloneqq (F\circ \gamma)'(0)
\end{align}
is a linear map between vector spaces, called differential or derivative of $F$ at $x$. 
This is illustrated in Fig. \ref{fig:differential}.

Let ${\mathcal X}({\mathcal M})$ denote the linear space of smooth vector fields on ${\mathcal M}$,
i.e., of smooth mappings from ${\mathcal M}$ to $T{\mathcal M}$.
On every Riemannian manifold there is the  Riemannian or  Levi-Civita connection
\[
	\nabla\colon T{\mathcal M} \times T{\mathcal M} \rightarrow T{\mathcal M}
\]
which is uniquely determined by the following properties:
\begin{itemize}
\item[i)] $\nabla_{f \Xi + g \Theta} X = f \nabla_\Xi  X + g \nabla_\Theta X$ for all $f,g \in {\mathcal C}^\infty({\mathcal M})$
and  all $\Xi , \Theta ,X \in {\mathcal X}({\mathcal M})$,
\item[ii)] $\nabla_\Xi(a X + b Y) = a \nabla_\Xi X + b \nabla_\Xi Y$ for all $a,b \in \mathbb R$ and $\Xi, X,Y \in {\mathcal X}({\mathcal M})$,
\item[iii)] $\nabla_\Xi (f X) = \Xi(f) X + f \nabla_\Xi X$ for all $f \in {\mathcal C}^\infty({\mathcal M})$ and all $\Xi,X \in {\mathcal X}({\mathcal M})$,
\item[iv)] $\nabla$ is compatible with the Riemannian metric $\langle \cdot, \cdot\rangle$, i.e., 
$\Xi \langle X,Y \rangle = \langle \nabla _{\Xi} X, Y \rangle + \langle X, \nabla _{\Xi} Y \rangle$,
\item[v)] $\nabla$ is symmetric (torsion-free): $[\Xi,\Theta] = \nabla_\Xi \Theta - \nabla_\Theta \Xi$,
where $[\cdot,\cdot]$ denotes the Lie bracket.
\end{itemize}
For some real interval $I$, a map $\Xi\colon I \rightarrow T{\mathcal M}$ is called a vector field along 
a curve $\gamma\colon I \rightarrow {\mathcal M}$ if $\Xi(t) \in T_{\gamma (t)} {\mathcal M}$ for all $t \in I$.
Let ${\mathcal X}(\gamma)$ denote the smooth vector fields along $\gamma$.
The Riemannian connection determines for each curve $\gamma\colon I \rightarrow {\mathcal M}$ 
a unique operator
$\frac{D}{dt}\colon {\mathcal X}(\gamma) \rightarrow {\mathcal X}(\gamma)$ with the properties:
\begin{itemize}
\item[T1)]
$\frac{D}{dt}(a \Xi + b \Theta ) = a \frac{D}{dt} \Xi + b \frac{D}{dt} \Theta$ for all $a,b \in \mathbb R$, 
\item[T2)]
$\frac{D}{dt} (f\Xi) = \dot f \Xi + f \frac{D}{dt} \Xi$ for all $f \in {\mathcal C}^\infty(I)$,
\item[T3)]
If $\Xi$ is extendible (to a neighborhood of the image of~$\gamma$), 
then for any extension $\tilde \Xi$ it holds
$\frac{D}{dt} \Xi(t) = \nabla_{\dot \gamma(t)} \tilde \Xi$.
\end{itemize}
Then $\frac{D}{dt} \Xi$ is called covariant derivative of $\Xi$ along $\gamma$.
A vector field $\Xi$ along a curve $\gamma$ is said to be parallel along $\gamma$ if 
\[
\frac{D}{dt} \Xi = 0.
\]
The fundamental fact about parallel vector fields is that any tangent
vector at any point on a curve can be uniquely extended to a parallel
vector field along the entire curve.
In this sense we can extend an (orthonormal) basis $\{\xi_1, \ldots,\xi_n\}$
of $T_x{\mathcal M}$ parallel along a curve $\gamma$ and
call this a parallel transported (orthonormal) frame $\{\Xi_1, \ldots,\Xi_n\}$ along $\gamma$.
Then any vector field $\Xi \in {\mathcal X}(\gamma)$ can be written as
$\Xi(t) = \sum_{j=1}^n a_j(t) \Xi(t)$ and we obtain by T1) and T2) that
\begin{align} \label{trans}
 \frac{D}{dt} \Xi  &= \frac{D}{dt} \big(\sum_{j=1}^n a_j \Xi_j\big) = \sum_{j=1}^n \frac{D}{dt} \left( a_j \Xi_j \right)
 =\sum_{j=1}^n \dot a_j(t)\Xi_j.
\end{align}
{\small
\bibliographystyle{abbrv}
\bibliography{references}
}
\end{document}